\documentclass[11pt,reqno]{amsart}
\usepackage[brazilian,english]{babel} %Hifenização
\usepackage[T1]{fontenc}
\usepackage{amsfonts,amsthm,amssymb,amsmath}  %Fontes matemática e símbolos
\usepackage{graphicx,color}       %Inserir gráficos
\usepackage{latexsym}       %Símbolos especiais
\usepackage{overpic}        %Inserir figuras e escrever com a fonte do documento
\usepackage[colorlinks=false, linktocpage=true]{hyperref}
\usepackage{graphicx}
\usepackage{multicol}
\usepackage{color}
\usepackage{multirow} %%%%%%%%%%%%
\newtheorem{teore}{Theorem}[section]
\newtheorem{obs}[teore]{Remark}%[section]
\newtheorem{defi}{Definition}[section]

\newtheorem{coro}[teore]{Corollary}%[section]

\newtheorem{lem}[teore]{Lemma}%[section]

\newtheorem{theorem}{}

\newcommand{\E}{\mathcal{E}}
\newcommand{\ka}{\kappa}
\newcommand{\R}{\mathbb{R}}
\newcommand{\N}{\mathbb{N}}
\newcommand{\C}{\mathbb{C}}

\newcommand{\D}{\mathcal{D}}
\newcommand{\Dz}{\dot{\mathbf{H}}^{1}(\R^{6})}
\newcommand{\psib}{\mbox{\boldmath$\psi$}}

\newcommand{\ub}{\mathbf{u}}
\newcommand{\vb}{\mathbf{v}}
\newcommand{\zb}{\mathbf{z}}

\numberwithin{equation}{section}

\title[Blow-up solutions for Schr\"{o}dinger systems ]{Blow-up solutions for a system of Schr\"{o}dinger equations with general quadratic-type nonlinearities in dimensions five and six}

\author[N. Noguera]{Norman Noguera}
\address{SM-UCR, Ciudad Universitaria Carlos Monge Alfaro, Departamento de Ciencias Naturales, Apdo: 111-4250, San Ram\'on, Alajuela, Costa Rica}
\email{norman.noguera@ucr.ac.cr}

\author[A. Pastor]{Ademir Pastor}
\address{IMECC-UNICAMP, Rua S\'ergio Buarque de Holanda, 651, 13083-859, Cam\-pi\-nas-SP, Bra\-zil}
\email{apastor@ime.unicamp.br}
\subjclass[2010]{35Q55, 35B44 35A01, 35J50.}

\keywords{Nonlinear Schr\"{o}dinger system; Quadratic-type interactions; Ground state solutions;  Blow-up; Energy-critical.}

\begin{document}
	
\begin{abstract}
In this work, we show the existence of ground state solutions for an $l$-component system of non-linear Schr\"{o}dinger equations with quadratic-type growth interactions in the energy-critical case.  They are obtained analyzing a critical Sobolev-type inequality and using the concentration-compactness method. As an application, we prove the existence of  blow-up  solutions of the system without the mass-resonance condition in  dimension six (and five), when the  initial data is radial. 
\end{abstract}

\maketitle

\tableofcontents

\section{Introduction}

This paper is concerned with the study of the following initial-value problem 
\begin{equation}\label{system1}
\begin{cases}
\displaystyle i\alpha_{k}\partial_{t}u_{k}+\gamma_{k}\Delta u_{k}-\beta_{k} u_{k}+f_{k}(u_{1},\ldots,u_{l})=0,\\
(u_{1}(x,0),\ldots,u_{l}(x,0))=(u_{10},\ldots,u_{l0}),\qquad k=1,\ldots l,
\end{cases}
\end{equation}
where $u_{k}:\R^{n}\times \R\to \C$, $(x,t)\in \R^{n}\times \R$, $\Delta$ is the Laplacian operator, $\alpha_{k}, \gamma_{k}>0$, $\beta_{k}\geq0$ are real constants and  the nonlinearities $f_{k}$ have a quadratic-type growth. 

Multi-component Schr\"odinger systems with quadratic interactions arise in many physical situations, for instance, in fiber and waveguide nonlinear optics (see \cite{Kiv} for a review and applications). Such systems may be obtained, for instance, by using the so-called multistep cascading mechanism. In particular, multistep cascading can be achieved by second-order nonlinear processes such as second harmonic generation (SHG) and sum-frequency mixing (SFM) (see \cite{KoSa}). An example of a three-step cascading model is
\begin{equation}\label{system1A}
\begin{cases}
\displaystyle 2i\partial_{t}w+\Delta w-\beta w=-\frac{1}{2}(u^{2}+v^{2}),\\
\displaystyle i\partial_{t}v+\Delta v-\beta_{1} v=-\overline{v}w,\\
\displaystyle i\partial_{t}u+\Delta u-u=- \overline{u}w,
\end{cases}
\end{equation}
which represents, in dimensionless variables, the reduced
amplitude equations of a fundamental beam with frequency $\omega$
entering a nonlinear medium with a quadratic response, derived in a slowly varying envelope approximation with the assumption of zero absorption of all interacting waves. Here $\beta, \beta_{1}\geq0$ are real constants and functions, $u,v$, and $w$ represents the associated polarizations. Another example is given by
\begin{equation}\label{system1B2}
\begin{cases}
\displaystyle i\partial_{t}w+\Delta w- w=-(\overline{w}v+\overline{v}u),\\
\displaystyle 2i\partial_{t}v+\Delta v-\beta v=-\left(\frac{1}{2}w^{2}+ \overline{w}u\right),\\
\displaystyle 3i\partial_{t}u+\Delta u-\beta_{1}u=-  vw,
\end{cases}
\end{equation}
where $w$, $v$, and $u$ represent, in dimensionless variables, the complex electric fields envelopes of the fundamental harmonic, second harmonic, and third harmonic, respectively (see \cite{Kiv} for details). A model formally appearing as a non-relativistic version  of some Klein-Gordon system, when the speed of light  tends to infinity is given, in dimensionless variables, by (see \cite{Hayashi})
\begin{equation}\label{system1J}
\begin{cases}
\displaystyle i\partial_{t}u+\Delta u=-2\overline{u}v,\\
\displaystyle i\partial_{t}v+\kappa\Delta v=- u^{2},
\end{cases}
\end{equation}
where $\kappa$ is a real constant. Similar systems can also be rigorously derived as a model in $\chi^{(2)}$ media (see \cite{Colin2}). 

From the mathematical point of view the interest in nonlinear Schr\"odinger systems with quadratic interactions has been increased in the past few years (see \cite{Colin}, \cite{Colin2}, \cite{Hamano}, \cite{Hayashi3}, \cite{Hayashi}, \cite{inui2018blow}, \cite{kishimoto2018scattering}, \cite{iwabuchi2016ill}, \cite{NoPa}, \cite{NoPa2}, \cite{Pastor2}, \cite{ozawa2013small} and references therein). So, in \cite{NoPa2} we initiated the study of system \eqref{system1} with  general quadratic-type nonlinearities. More precisely, we assumed the following (with a slight modification in \ref{H4}).
\renewcommand\thetheorem{(H1)}
\begin{theorem}\label{H1}
 \begin{align*}
f_{k}(0,\ldots,0)=0, \qquad  k=1,\ldots,l. 
\end{align*}
\end{theorem}

\renewcommand\thetheorem{(H2)}
\begin{theorem}\label{H2}
 There exists a constant $C>0$ such that for $(z_{1},\ldots,z_{l}),(z_{1}',\ldots,z_{l}')\in \C^{l}$ we have 
\begin{equation*}
\begin{split}
\left|\frac{\partial }{\partial z_{m}}[f_{k}(z_{1},\ldots,z_{l})-f_{k}(z_{1}',\ldots,z_{l}')]\right|&\leq C\sum_{j=1}^{l}|z_{j}-z_{j}'|,\qquad k,m=1,\ldots,l;\\
\left|\frac{\partial }{\partial \overline{z}_{m}}[f_{k}(z_{1},\ldots,z_{l})-f_{k}(z_{1}',\ldots,z_{l}')]\right|&\leq C\sum_{j=1}^{l}|z_{j}-z_{j}'|,\qquad k,m=1,\ldots,l.
\end{split}
\end{equation*}
\end{theorem}

\renewcommand\thetheorem{(H3)}
\begin{theorem}\label{H3}
There exists a function $F:\C^{l}\to \C$ such that    
\begin{equation*}
f_{k}(z_{1},\ldots,z_{l})=\frac{\partial F}{\partial \overline{z}_{k}}(z_{1},\ldots,z_{l})+\overline{\frac{\partial F }{\partial z_{k}}}(z_{1},\ldots,z_{l}),\qquad k=1\ldots,l. 
\end{equation*}
\end{theorem}

\renewcommand\thetheorem{(H4)}
\begin{theorem}\label{H4}
 There exist positive constants $\sigma_{1},\ldots,\sigma_{l}$ such that for any $(z_{1},\ldots,z_{l})\in \mathbb{C}^{l}$
\begin{equation*}
\mathrm{Im}\sum_{k=1}^{l}\sigma_{k}f_{k}(z_{1},\ldots,z_{l})\overline{z}_{k}=0 .
\end{equation*}	
\end{theorem}

\renewcommand\thetheorem{(H5)}
\begin{theorem}\label{H5}
Function $F$ is homogeneous of degree 3, that is, for any $\lambda >0$  and $(z_{1},\ldots,z_{l})\in \mathbb{C}^{l}$,
\begin{equation*}
F(\lambda z_{1},\ldots,\lambda z_{l})=\lambda^{3}F(z_{1},\ldots,z_{l}).
\end{equation*}
\end{theorem}

\renewcommand\thetheorem{(H6)}
\begin{theorem}\label{H6}
There holds
\begin{equation*}
\left|\mathrm{Re}\int_{\R^{n}} F(u_{1},\ldots,u_{l})\;dx\right|\leq \int_{\R^{n}} F(|u_{1}|,\ldots,|u_{l}|)\;dx.  
\end{equation*}
\end{theorem}

\renewcommand\thetheorem{(H7)}
\begin{theorem}\label{H7}
Function $F$ is real valued on $\R^l$, that is, if $(y_{1},\ldots,y_{l})\in \R^{l}$ then
\begin{equation*}
F(y_{1},\ldots,y_{l})\in \R. 
\end{equation*}
Moreover, functions	$f_k$ are non-negative on the positive cone in $\mathbb{R}^l$, that is, for $y_i\geq0$, $i=1,\ldots,l$,
\begin{equation*}
f_{k}(y_{1},\ldots,y_{l})\geq0.
\end{equation*}

\end{theorem}

\renewcommand\thetheorem{(H8)}
\begin{theorem}\label{H8}
	Function $F$ can be written as the sum $F_1+\cdots+F_m$, where $F_s$, $s=1,\ldots, m$ is super-modular on $\R^d_+$, $1\leq d\leq l$, and vanishes on hyperplanes, that is, for any $i,j\in\{1,\ldots,d\}$, $i\neq j$ and $k,h>0$, we have
	\begin{equation*}
	F_s(y+he_i+ke_j)+F_s(y)\geq F_s(y+he_i)+F_s(y+ke_j), \qquad y\in \R^d_+,
	\end{equation*}
	and $F_s(y_1,\ldots,y_d)=0$ if $y_j=0$ for some $j\in\{1,\ldots,d\}$.
\end{theorem}

It is easy to see that functions $F$ associated to systems \eqref{system1A}, \eqref{system1B2}, and \eqref{system1J} are given, respectively, by
\begin{equation}\label{Fsys}
F(z_1,z_2,z_3)=\frac{1}{2}\overline{z}_1(z_2^2+z_3^2), \quad F(z_1,z_2,z_3)=\frac{1}{2}z_1^2\overline{z}_2+z_1z_2\overline{z}_3, \quad F(z_1,z_2)=\overline{z}_1^2z_2.
\end{equation}
In addition assumptions \ref{H1}-\ref{H8} hold in these particular examples.

The results established in \cite{NoPa2} include local and global well posedness in $L^2(\R^n)$ and $H^1(\R^n)$, $1\leq n\leq 6$, existence and stability/instability of ground state solutions, and the dichotomy global existence versus blow up in finite time. In particular, assumptions \ref{H1} and \ref{H2} are enough to prove the existence of a local solution by using the contraction mapping principle in a suitable space based on the well-known Strichartz estimates. Assumptions \ref{H3}-\ref{H5} guarantee that \eqref{system1} conserves the charge
	\begin{equation}\label{Charge}
Q(\ub(t)):=\sum_{k=1}^{l}\frac{\alpha_{k}\sigma_{k}}{2}\|  u_{k}(t)\|_{L^{2}}^{2},
\end{equation}
and  the energy
\begin{equation}\label{Energy}
E(\ub(t)):=\sum_{k=1}^{l}\gamma_{k}\|\nabla u_{k}(t)\|_{L^2}^{2}+\sum_{k=1}^{l}\beta_{k}\|u_{k}(t)\|_{L^2}^{2}
-2\mathrm{Re}\int F(\ub(t))\;dx
\end{equation}
where we are using the notation $\ub(t)=(u_1(t),\ldots,u_l(t))$ (see notations below). By using the above conserved quantities one can show the existence of global solutions in $L^2(\R^n)$ and $H^1(\R^n)$, $1\leq n\leq3$. Also, if the $H^1$-norm of the initial data is sufficiently small one can also show the global existence in $H^1(\R^n)$ if $n=4$ or $n=5$. Moreover, when  \ref{H6}-\ref{H8} are assumed we proved the existence and stability/instability of \textit{ground state solutions}. Using this especial kind of solutions we were able to provide a sharp vectorial Gagliardo-Nirenberg-type inequality to give a sharp criterion for the existence of global solutions in dimensions $n=4$ and $n=5$. Some of the above results are summarized in Section \ref{knowresults}.

Before presenting the main goal of this paper, we recall that by a standard scaling argument and the fact that $f_{k}$ are homogeneous functions of degree 2 (see \eqref{fkhomog2}) we deduce that  $\dot{H}^{n/2-2}(\R^n)$ is the  critical (scaling invariant) Sobolev space for \eqref{system1} (with $\beta_{k}=0$). In particular, $L^{2}$ and $\dot{H}^{1}$ are critical in dimensions $n=4$ and $n=6$, respectively. As usual, we then adopt the following convention: we will say that system \eqref{system1} is
 \begin{equation*}
     L^{2}-
     \begin{cases}
     \mbox{subcritical}, \quad\mbox{if}\quad 1\leq n\leq 3,\\
     \mbox{critical},\quad\mbox{if}\quad n=4,\\
     \mbox{supercritical},\quad\mbox{if}\quad n\geq5,
     \end{cases}\quad \mbox{and} \quad
     H^{1}-
     \begin{cases}
     \mbox{subcritical},\quad\mbox{if}\quad 1\leq n\leq 5,\\
     \mbox{critical},\quad\mbox{if}\quad n=6,\\
     \mbox{supercritical},\quad\mbox{if}\quad n\geq 7.
     \end{cases}
 \end{equation*}

To proceed we introduce the following  definition
 \begin{defi}\label{defmasres}  We say that  
 	\eqref{system1}  satisfies the mass-resonance condition if
 	\begin{equation}\label{RC}
 	\mathrm{Im}\sum_{k=1}^{l}m_{k}f_{k}(\zb)\overline{z}_{k}=0, \quad \zb\in \mathbb{C}^{l},\tag{RC}
 	\end{equation}
 	where $m_{k}:=\frac{\alpha_{k}}{2\gamma_{k}}$.
 \end{defi}

 Let us  illustrate Definition \ref{defmasres} using our examples above. We first check that \eqref{system1J} satisfies the mass-resonance condition if and only if $\ka=\frac{1}{2}$, which is in agreement with the terminology in the current literature. Indeed, as we already said, the function $F$ associated to \eqref{system1J} is $F(z_1,z_2)=\overline{z}_1^2z_2$ and \eqref{RC} is equivalent to
 \begin{equation*}
 \left(1-\frac{1}{2\kappa}\right)\mathrm{Im}(\overline{z}_1^2z_2)=0,
 \end{equation*}
 which means that mass-resonance occurs only if $\kappa=\frac{1}{2}$. On the other hand, using \eqref{Fsys},
 it is easy to see that systems \eqref{system1A} and \eqref{system1B2} both satisfy  \eqref{RC}.

 	We point out that instead of \ref{H4}, in \cite{NoPa2} we have assumed
 	\begin{equation}\label{antH4}
 	\mathrm{Re}\,F\left(e^{i\frac{\alpha_{1}}{\gamma_{1}}\theta  }z_{1},\ldots,e^{i\frac{\alpha_{l}}{\gamma_{l}}\theta  }z_{l}\right)=\mathrm{Re}\,F(\mathbf{z}), \qquad \theta\in\R,\;\mathbf{z}=(z_1,\ldots,z_l),
 	\end{equation}	
 	which, together with \ref{H3}, implies that \eqref{RC} holds (see Lemma 2.9 in \cite{NoPa2}). Thus, all results obtained in \cite{NoPa2} are under the assumption of mass-resonance. It is our goal in in the present paper to study \eqref{system1} without the assumption of mass-resonance.

Mass-resonance  appears as a special relation between the masses of the system and it is closely related with the large time behavior of solutions.  As pointed out in \cite{uriya2016final}, it was first considered in Klein-Gordon systems. When considering system \eqref{system1J}, for instance,  it is well known that the value of the parameter $\kappa$  influences the large-time behavior of its solutions, see \cite{iwabuchi2016ill}.  In \cite{Hayashi} the authors,  among other things, proved the existence of ground state solutions for  \eqref{system1J}, when $\kappa>0$, and used these solutions to give a sharp criterion for the existence of global  $H^1$ solutions in the mass-resonance case and $n=4$.  This kind of result was extended  to the non-mass-resonance case ($\ka\neq\frac{1}{2}$) in  \cite{inui2018blow}, where  the authors  showed a blow-up result when the initial data is radial in dimensions $n=5$ and $n=6$.  Some very recent works  without mass-resonance condition  have been appeared. In \cite{kishimoto2018scattering}, the authors showed scattering in the $L^2$-critical case  with and without the mass-resonance condition. Moreover, scattering in the case $n=5$ was proved in \cite{hamano2019scattering} and in \cite{wangyang}. 

From the mathematical point of view, the phenomenon of mass-resonance can be seen in the virial-type identity satisfied by solutions of system \eqref{system1}. Indeed,  for $1\leq n\leq 6$,   set $\Sigma=\{\ub\in \mathbf{H}^{1}; x\ub\in \mathbf{L}^{2} \},$
where $x\ub$ means $(xu_{1},\ldots,xu_{l})$, and define the function
\begin{equation}\label{fuctV}
V(t)=\sum_{k=1}^{l}\frac{\alpha_{k}^{2}}{\gamma_{k}}\int |x|^{2}|u_{k}(x,t)|^{2}\;dx,
\end{equation}
where $\ub(t)$ is the corresponding solution of  \eqref{system1} with initial data $\ub_{0}\in \Sigma$. Then, if $I$ is the maximal existence interval,  a straightforward computation leads to
\begin{equation}\label{V2mwmr}
\begin{split}
V''(t)&=2nE(\ub_{0})-2n\sum_{k=1}^l\beta_k\|u_k\|_{L^2}^2+2(4-n)\sum_{k=1}^l\gamma_k\|\nabla u_k\|_{L^2}\\
&\quad -4\frac{d}{dt}\left[\int|x|^{2}\mathrm{Im}\sum_{k=1}^{l}m_{k}f_{k}(\mathbf{u})\overline{u}_{k}\;dx\right],
\end{split}
\end{equation}
for all $t\in I$.
Assuming that \eqref{RC}  holds, the last term in \eqref{V2mwmr}  disappears. In that case, in \cite{NoPa2}  it was shown  using an argument due to Glassey that if $E(\ub_{0})<0$ (or $E(\ub_{0})=0$ and $\ub_{0}$ has negative momentum), the local solution  blows-up in finite time in dimensions $4\leq n\leq 6$. 
The mass-resonance assumption has also been appeared in various  works involving  two and three-component Schr\"{o}dinger systems (see for instance \cite{ogawa2015final}, \cite{uriya2016final}, \cite{ozawa2013small}, \cite{hoshino2013analytic}, \cite{hoshino2015analytic}, \cite{hoshino2015analytic2} and \cite{hoshino2017analytic} and references therein). When \eqref{RC} does not hold a more careful analysis must be performed and we do not know if solutions in $\Sigma$ with negative energy, for instance, blow-up or not.

Based on the above background, the main goal of this paper is  to prove that blow-up in finite time also holds if mass-resonance does not occur, but  the initial data is radial. We will be particularly interested in the cases $n=5$ and $n=6$. The ``threshold'' for the existence of blow-up solutions will be given in terms of the ground states associated with \eqref{system1}. See Theorem \ref{thm:Blowupn6} below.

This work is organized as follows. In section \ref{sec.prel} we first introduce some notations and give preliminaries results that  will be used along the paper. We also list some consequences of our assumptions and give a review of  some previous results about system  \eqref{system1}.
In section \ref{sec.exigsn=6} we use the concentration-compactness method to prove the existence of ground state solutions in the $H^1$-critical case. We also establish the optimal constant in a critical Sobolev-type inequality. Finally, section \ref{sec.blowupn56} is devoted to show that in dimensions $n=5$ and $n=6$ if the initial data is radially symmetric then the corresponding solution of \eqref{system1}  blows-up in finite time.

%%%%%%%%%%%%%%%%%%%%%%%%%%%%%%%%%%%%%%%%

%%%%%%%%%%%%%%%%%%%%%%%%%%%%%%%%%%%%%%%%%%%%%%%%%%%

\section{Preliminaries}\label{sec.prel}
In this section we introduce some notations and give some consequences of our assumptions.
\subsection{Notation}
 We use $C$ to denote several constants that may vary line-by-line. $B(x,r)$ denotes the ball of radius $r$ and center at  $x\in \R^{n}$.  To simplify writing, given any set $A$, by  $\mathbf{A}$ (or $A^{l}$) we denote  the product  $\displaystyle A\times \cdots \times A $ ($l$ times). If $A$ is a Banach space with norm $\|\cdot\|$ then $\mathbf{A}$ is a Banach space with  norm given by the sum. Thus,  in $\mathbb{C}^l$  we use frequently $\mathbf{z}$ instead of  $(z_{1},\ldots,z_{l})$.  Given any complex number $z\in\mathbb{C}$, $\mathrm{Re}z$ and $\mathrm{Im}z$ represents its real and imaginary parts. Also, $\overline{z}$ denotes its complex conjugate.  We set $\big\bracevert\!\! \mathbf{z}\!\!\big\bracevert$  for the vector $(|z_{1}|,\ldots,|z_{l}|)$. This is not to be confused with $|\mathbf{z}|$ which stands for the standard norm of the vector $\mathbf{z}$ in $\mathbb{C}^l$. If $\mathbf{w}$ is a vector with non-negative  real components, we write  $\mathbf{w}\geq \mathbf{0}$.  Given $\mathbf{z}=(z_{1},\ldots,z_{l})\in \mathbb{C}^l$, we write $z_m=x_m+iy_m$, where $x_m=\mathrm{Re}z_m$ and $y_m=\mathrm{Im}z_m$. The differential operators $\partial/\partial z_m$ and $\partial/\partial \overline{z}_m$ are defined by
$$
\dfrac{\partial}{\partial z_m}=\frac{1}{2}\left(\frac{\partial}{\partial x_m} -i\frac{\partial}{\partial y_m}\right), \qquad\dfrac{\partial}{\partial \overline{z}_m}=\frac{1}{2}\left(\frac{\partial}{\partial x_m} +i\frac{\partial}{\partial y_m}\right).
$$ 

Let $\Omega\subset \R^{n}$ be an open set.
To simplify notation, if no confusion is caused we use $\int f\, dx$ instead of  $\int_{\Omega} f\, dx$. The spaces
 $L^{p}=L^{p}(\Omega)$, $1\leq p\leq \infty$, and $W^s_p=W^s_p(\Omega)$ denote the standard Lebesgue and Sobolev spaces. In the case $p=2$, we use the notation $H^s=W^s_2$.  We use $\dot{H}^1=\dot{H}^1(\R^n)$ to denote the homogeneous Sobolev spaces of order 1. For $n\geq3$, $2^*=\frac{2n}{n-2}$ denotes the critical Sobolev exponent. 
 Recall that $D^{1,2}(\Omega)=\{u\in L^{2^{*}}(\Omega); \nabla u \in L^{2}(\Omega)\}$ and $D_{0}^{1,2}(\Omega)$  is	the completation of $\mathcal{C}_{0}^{\infty}(\Omega)$ 
 	with respect to the norm $\left(\int|\nabla u|^{2}\;dx\right)^{\frac{1}{2}}$, or equivalently,  the closure of $\mathcal{C}_{0}^{\infty}(\Omega)$ in $D^{1,2}(\Omega)$. By the Sobolev inequality we have  $D^{1,2}(\R^n)=\dot{H}^{1}(\R^n)$ (with equivalent norms). Since   $D^{1,2}(\R^{n})=D^{1,2}_{0}(\R^{n})$ (see for instance \cite[Lemma 1.2]{ben1996extrema}), we then see that
 	$$
 \dot{H}^{1}(\R^n)=	D^{1,2}(\R^{n})=D^{1,2}_{0}(\R^{n}).
 	$$
Thus we can use each one of these spaces in our arguments to follow.

Given a time interval $I$, the mixed   spaces $L^p(I;L^q(\R^n))$ are endowed with the norm
$$
\|f\|_{L^p(I;L^q)}=\left(\int_I \left(\int_{\R^n}|f(x,t)|^qdx \right)^{\frac{p}{q}} dt \right)^{\frac{1}{p}},
$$
with the obvious modification if either $p=\infty$ or $q=\infty$. When the interval $I$ is implicit and no confusion will be caused we denote  $L^p(I;L^q(\R^n))$ simply by  $L^p(L^q)$ and its norm by $\|\cdot\|_{L^p(L^q)}$. More generally, if $X$ is a Banach space, $L^{p}(I;X)$ represents the $L^p$ space of $X$-valued functions defined on $I$. 

With  $\mathcal{C}_{b}(X)$ we denote the set of bounded continuous functions on $X$. Also, $\mathcal{C}_{c}(X)$ stands for the set of continuous functions on $X$ with compact support. The set  $\mathcal{M}_{+}(X)$ denotes the Banach space of non-negative Radon measures on $X$. Similarly,   $\mathcal{M}_{+}^{b}(X)$ and $\mathcal{M}_{+}^{1}(X)$ represent the spaces of bounded (or finite) and probability Radon measures, respectively. By  $\mathcal{B}(X)$ we denote  the Borel $\sigma$-algebra on $X$. We  write $\nu\ll \mu$ if the measure $\nu$ is absolutely continuous with respect to the measure $\mu$. For any $\mu\in\mathcal{M}_{+}^{b}(X)$, $\|\mu\|:=\mu(X)$ is called the \textit{total mass} of $\mu$.

\subsection{Weak convergence of measures} Here we introduce some notions of convergence of Radon measures. We refer the reader to \cite[Chapter 4, Sections 30-31]{bauer2011measure}  for a more complete discussion about this topic. Let $X$ be a locally compact space. 
A sequence $(\mu_{m})\subset\mathcal{M}_{+}(X)$ is said to converge \textit{vaguely}  to  $\mu  $ in $\mathcal{M}_{+}(X)$, written $\mu_{m}\overset{\ast}{\rightharpoonup} \mu$, provided $\int_{X}f\;d\mu_{m}\to \int_{X}f\;d\mu $, for all $f\in \mathcal{C}_{c}(X)$.  We say that a sequence $(\mu_{m})\subset\mathcal{M}_{+}^{b}(X)$  converges \textit{weakly} to a measure $\mu$ in   $\mathcal{M}_{+}^{b}(X)$, written $\mu_{m} \rightharpoonup \mu$,  if  $\int_{X}f\;d\mu_{m}\to \int_{X}f\;d\mu $, for all $f\in \mathcal{C}_{b}(X)$. A sequence $(\mu_{m}) \subset \mathcal{M}_{+}^{b}(X) $ is said to be \textit{uniformly tight} if, for every $\epsilon>0$ there exists a compact subset $K_{\epsilon}\subset X$ such that $\mu_{m}(X\setminus K_{\epsilon})< \epsilon$, for all $m$. 

We finish this paragraph with an useful result that guarantees the existence of vaguely convergent subsequences.  We say that a set $\mathcal{H}\subset \mathcal{M}_{+}(X)$ is \textit{vaguely bounded} if $\sup_{\mu\in \mathcal{H} }\left|\int_{X}f\;du\right|<\infty$ for all $f\in \mathcal{C}_{c}(X)$.

  \begin{lem}\label{vagbndconv}
 Let $X$ be  a locally compact space. Then,
 
 \begin{itemize}
 	\item[(i)] every vaguely bounded sequence in $\mathcal{M}_{+}(X)$ contains a vaguely convergent subsequence;
 	\item[(ii)] If $\mu_{m}\overset{\ast}{\rightharpoonup} \mu$ in $\mathcal{M}_{+}(X)$ and $(\|\mu_m\|)$ is bounded, then $\mu$ is finite.
 \end{itemize} 
 \end{lem}
\begin{proof}
See Theorems 31.2 and 30.6 in \cite{bauer2011measure}.
\end{proof}

\subsection{Some consequences of our assumptions} Here, we will present some consequences of our assumptions \ref{H1}-\ref{H8}. We start with the following.

\begin{lem}\label{estdifF} 
Assume that  \textnormal{\ref{H1}-\ref{H7}}   hold.  
\begin{itemize}
	\item[(i)] We have
	\begin{equation}\label{estdifFeq}
	\begin{split}
	\left|\mathrm{ Re}\,F(\mathbf{z})-\mathrm{ Re}\,F(\mathbf{z}')\right|&\leq  C \sum_{m=1}^{l}\sum_{j=1}^{l}(|z_{j}|^{2}+|z_{j}'|^{2})|z_{m}-z_{m}'|.
	\end{split}
	\end{equation}
	In particular,
	\begin{equation*}
	|\mathrm{Re}\,F(\mathbf{z})|\leq C \sum_{j=1}^{l}|z_{j}|^{3}.
	\end{equation*}
	\item[(ii)] Let $\mathbf{u}$ be a complex-valued function defined on $\R^n$. Then, 
	\begin{equation*}
	\mathrm{Re}\sum_{k=1}^{l}f_{k}(\mathbf{u})\overline{u}_{k}=\mathrm{Re}[3F(\mathbf{u})].
	\end{equation*}
	\item[(iii)] We have
	\begin{equation}\label{fkreal}
	f_{k}({x})=\frac{\partial F}{\partial x_{k}}({x}).%\tag{H3*}
	\end{equation}
	In addition,  $F$ is positive on the positive cone of $\R^l$.
\end{itemize}
 \end{lem}
 \begin{proof}
 The reader will find the details in \cite{NoPa2}. More precisely, see Lemmas 2.10, 2.11, and 2.13.
 \end{proof}

 Now by using assumptions \ref{H3} and \ref{H4} we are able  to derive  a Gauge invariant condition  satisfied by the non-linear interaction terms in \eqref{system1}.
 We start with the following invariant property of $\mathrm{Re}\,F$.
 
 \begin{lem}\label{ReFinvari2}
 Assume that \textnormal{\ref{H3}} and  \textnormal{\ref{H4}} hold. Let $\theta \in \R$ and $\zb \in \C^{l}$. Then, 
     \begin{equation*}
\mathrm{Re}\,F\left(e^{i\frac{\sigma_{1}}{2}\theta  }z_{1},\ldots,e^{i\frac{\sigma_{l}}{2}\theta  }z_{l}\right)=\mathrm{Re}\,F(\zb).
\end{equation*}
\end{lem}
\begin{proof}
Denote by  $\mathbf{w}$ the vector $(w_{1},\ldots,w_{l}):=\left(e^{i\frac{\sigma_{1}}{2}\theta  }z_{1},\ldots,e^{i\frac{\sigma_{l}}{2}\theta  }z_{l}\right)$ and let  $h(\theta):=F\left(\mathbf{w}\right)$. By the chain rule,
\begin{equation}
	\begin{split}\label{derh}
	\frac{dh}{d\theta}&=\sum_{k=1}^{l}\frac{\partial F}{\partial w_{k}}(\mathbf{w})\frac{\partial w_{k} }{\partial \theta}+\sum_{k=1}^{l}\frac{\partial F }{\partial \overline{w}_{k}}(\mathbf{w})\frac{\partial \overline{w}_{k} }{\partial \theta}\\
	&=\sum_{k=1}^{l}\frac{\partial F}{\partial w_{k}}(\mathbf{w})\left(\frac{\sigma_{k}}{2} i\right)e^{i\frac{\sigma_k}{2}\theta  }z_{k}+\sum_{k=1}^{l}\frac{\partial F }{\partial \overline{w}_{k}}(\mathbf{w})\left(-\frac{\sigma_{k}}{2} i\right)\overline{e^{i\frac{\sigma_{k}}{2}\theta  }z_{k}}\\
	&=\sum_{k=1}^{l}\overline{\overline{\frac{\partial F}{\partial w_{k}}}(\mathbf{w})\left(-\frac{\sigma_{k}}{2} i\right)\overline{w}_{k}}+\sum_{k=1}^{l}\frac{\partial F }{\partial \overline{w}_{k}}(\mathbf{w})\left(-\frac{\sigma_{k}}{2} i\right)\overline{w}_{k}.
	\end{split}
\end{equation}
Taking the real part on both sides  of \eqref{derh} and using \ref{H3} we obtain
\begin{equation*}
	\begin{split}
	\mathrm{Re}\frac{dh}{d\theta}
	&=\frac{1}{2}\mathrm{Im}\sum_{k=1}^{l} \sigma_k f_{k}(\mathbf{w})\overline{w}_k=0,
	\end{split}
\end{equation*}
which implies the desired.
\end{proof}  

With the previous result in hand we prove the following.
\begin{lem}\label{H34impGC}
Under the assumptions of Lemma \ref{ReFinvari2}. The functions $f_{k}$, $k=1,\ldots, l$, satisfy  the following  Gauge   condition
\begin{equation}\label{gaugeCon}
f_{k}\left(e^{i\frac{\sigma_{1}}{2}\theta }z_{1},\ldots,e^{i\frac{\sigma_{l}}{2}\theta }z_{l}\right)=e^{i\frac{\sigma_{k}}{2}\theta }f_{k}(\mathbf{z}). \tag{GC}
\end{equation}
\end{lem}
\begin{proof}
	First of all note that from the definition of the complex differential operators we may write
	$$
	f_k(\zb) =2\frac{\partial  }{\partial \overline{z}_{k}}\mathrm{Re} \,F(\mathbf{z}).
	$$
Now, as in the proof of Lemma \ref{ReFinvari2}, let $\mathbf{w}=\left(e^{i\frac{\sigma_{1}}{2}\theta  }z_{1},\ldots,e^{i\frac{\sigma_{l}}{2}\theta  }z_{l}\right)$. Clearly,
\begin{equation}\label{dev0}
\frac{\partial }{\partial \overline{z}_{k}}\mbox{Re}\, F(\mathbf{z})=e^{-i\frac{\sigma_{k}}{2}\theta  }\frac{\partial}{\partial \overline{w}_{k}} \mbox{Re}\,F(\mathbf{w}).
\end{equation}
Hence,
\begin{equation*}
\begin{split}
f_{k}(\mathbf{z})=2e^{-i\frac{\sigma_{k}}{2}\theta}\frac{\partial}{\partial \overline{w}_{k}} \mbox{Re}\,F(\mathbf{w})
=e^{-i\frac{\sigma_{k}}{2}\theta  }f_{k}(\mathbf{w})
=e^{-i\frac{\sigma_{k}}{2}\theta  }f_{k}\left(e^{i\frac{\sigma_{1}}{2}\theta  }z_{1},\ldots,e^{i\frac{\sigma_{l}}{2}\theta  }z_{l}\right),
\end{split}
\end{equation*}
which completes the proof.
\end{proof}

The next fact  is a natural consequence of assumptions \ref{H3}  and \ref{H5}. The nonlinearities $f_{k}$ are homogeneous functions of degree 2, i.e., for any $\mathbf{z}\in \C^{l}$,
\begin{equation}\label{fkhomog2}
    f_{k}(\lambda \mathbf{z})=\lambda^{2} f_{k}( \mathbf{z}), \qquad\forall k=1,\ldots,l, \qquad \lambda>0. 
\end{equation}

We finish this section by presenting an adapted vectorial version of the generalized Brezis-Lieb's lemma (see   \cite[Theorem 2]{brezis1983relation}). We start recalling that  by assumption \ref{H5}, $F(\mathbf{0})=0$ and, for all $\mathbf{a}, \mathbf{b}\in \C^{l}$,
\begin{equation*}
    |F(\mathbf{a}+\mathbf{b})-F(\mathbf{b})|\leq C\sum_{k=1}^{l}\sum_{j=1}^{l}\left(|a_{j}+b_{j}|^{2}+|b_{j}|^{2}\right)|b_{k}|. 
\end{equation*}
In particular $F$ is  continuous and, by
Young's inequality, for any $\epsilon>0$,
\begin{equation}\label{FBL}
  |F(\mathbf{a}+\mathbf{b})-F(\mathbf{b})|\leq \epsilon \varphi(\mathbf{a}) +\psi_{\epsilon}(\mathbf{b}),
\end{equation}
where $\varphi$ and $\psi_{\epsilon}$ are given by the non-negative functions $\varphi(\mathbf{a})=\sum_{j=1}^{l}|a_{j}|^{3}$ and
$\psi_{\epsilon}(\mathbf{b})=C_{\epsilon}\sum_{j=1}^{l}|b_{j}|^{3}$, for some positive constant $C_{\epsilon}$.

\begin{lem}\label{BLLG} Let $\vb_{m}=\ub_{m}-\ub$ be a sequence of measurable functions from  $\R^n$ to $\C$ such that
\begin{enumerate}
    \item[(i)] $\vb_{m}\to \mathbf{0}$ a.e.;
    \item[(ii)] $F(\ub)\in L^{1}(\R^n)$;
    \item[(iii)] $\int \varphi(\vb_{m})(x)\; dx\leq M<\infty$, for some constant $M$, independent of $\epsilon$ and $m$;
    \item[(iv)] $\int \psi_{\epsilon}(\ub)(x)\; dx<\infty$, for any $ \epsilon>0$. 
\end{enumerate}
 Then, as $m\to \infty$,
 \begin{equation*}
     \int |F(\ub_{m})-F(\vb_{m})-F(\ub)|\;dx \to 0.
 \end{equation*}
 \end{lem}
\begin{proof}
The proof is similar to that of Theorem 2 in \cite{brezis1983relation}. So we omit the details.
\end{proof}

 \subsection{Local and global well-posedness}  \label{knowresults}
  In \cite{NoPa2} we studied \eqref{system1} by assuming \ref{H1}-\ref{H8}, but with \ref{H4} replaced by \eqref{antH4}. From the point of view of well-posedness in $H^1$ the same results can also be obtained here in dimension $1\leq n\leq6$. Indeed, to give a precise statement we introduce the space
\begin{equation*}
Y(I):=\begin{cases}
(\mathcal{C}\cap L^{\infty})(I; H^1)\cap L^{12/n}\left(I;W^{1}_{3}\right),\qquad 1\leq n\leq 3,\\
(\mathcal{C}\cap L^{\infty})(I; H^1)\cap L^{2}(I;W^{1}_{2n/(n-2)}),\qquad n\geq 4,
\end{cases}
\end{equation*}
where  $I\subset\R$ is an interval. We have the following results. 
 
 \begin{teore}\label{localexistenceH1} Let $1\leq n\leq 5$. Assume that \textnormal{\ref{H1}} and \textnormal{\ref{H2}} hold. Then for any $r>0$ there exists $T(r)>0$ such that for any $\mathbf{u}_0\in \mathbf{H}^1 $ with $\|\mathbf{u}_0\|_{\mathbf{H}^1}\leq r$,  system \eqref{system1}
 has a unique  solution $\mathbf{u}\in \mathbf{Y}(I)$ with $I=[-T(r),T(r)]$.
 
 In addition, a  blow up alternative also holds, that is, there exist $T_*,T^*\in(0,\infty]$ such that the local solutions can be extend to  $(-T_*,T^*)$.  Moreover if $T_*<\infty$  (respect. $T^*<\infty$), then
$$
\lim_{t\to -T_*}\|\mathbf{u}(t)\|_{\mathbf{H}^1}=\infty, \qquad (respect.\lim_{t\to T^*}\|\mathbf{u}(t)\|_{\mathbf{H}^1}=\infty  ).
$$
\end{teore}
\begin{proof}
	See  \cite[Theorem 3.9]{NoPa2}.
\end{proof}

\begin{teore}\label{locexistH1n=6} Let $n=6$. Assume that  \textnormal{\ref{H1}} and \textnormal{\ref{H2}} hold. Then for any  $\mathbf{u}_0\in \mathbf{H}^1 $ there exists $T(\mathbf{u}_0)>0$ such that  system \eqref{system1}
 has a unique  solution $\mathbf{u}\in\mathbf{Y}(I)$ with $I=[-T(\mathbf{u}_0),T(\mathbf{u}_0)]$.
 In addition, a  blow up alternative also holds, that is, there exist $T_*,T^*\in(0,\infty]$ such that the local solution can be extend to  $(-T_*,T^*)$.  Moreover if $T_*<\infty$  (respect. $T^*<\infty$), then
$$
\lim_{t\to -T_*}\| \mathbf{u}(t)\|_{\mathbf{L}^{q}(W^{1}_{r})}=\infty, \qquad (respect.\lim_{t\to T^*}\|\mathbf{u}(t)\|_{\mathbf{L}^{q}(W^{1}_{r})}=\infty  ),
$$   
 for any (Schr\"odinger admissible) pair $(q,r)$ satisfying
 $$
 \frac{2}{q}=6\left(\frac{1}{2}-\frac{1}{r}\right), \qquad 2\leq r\leq 3.
 $$
\end{teore}
 \begin{proof}
 	See \cite[Theorem 3.10]{NoPa2} for the local well-posedness. The blow-up alternative can be established by extending the arguments in \cite[Theorem 4.5.1]{Cazenave}. 
 \end{proof}

Note that both results above hold only under assumptions {\ref{H1}} and {\ref{H2}}. To extend the local solutions to global ones we need  \ref{H3} and \ref{H4} to guarantee that the quantities \eqref{Charge} and \eqref{Energy} are conserved by the flow of \eqref{system1}. At this point, it is to be noted that in order to establish de conservation of $Q$, \eqref{antH4} may indeed be replaced by  \ref{H4} (see \cite[Lemma 3.11]{NoPa2} for details). Actually \ref{H4} is a vectorial extension for the well known assumptions for the scalar Schr\"odinger equation with a general nonlinearity (see  \cite[Chapter 3]{Cazenave}).

Using the above mentioned conserved quantities combined with the Gagliardo-Nirenberg inequality it is possible to get an \textit{a priori} bound for the $L^{2}$ and $H^{1}$-norm of a solution. In particular,  for any initial data in  $H^1$ local solutions can be extended to global ones, when $1\leq n\leq 3$. For  $n=4$ and $n=5$ global solutions are obtained if the initial data is sufficiently small. To give a precise description of how small the initial data must be, the ground states solutions take a singular place. In fact, a  standing wave solution for \eqref{system1} is a special solution having the form
\begin{equation}\label{standing}
u_{k}(x,t)=e^{i\frac{\sigma_{k}}{2}\omega t}\psi_{k}(x),\qquad k=1,\ldots,l,
\end{equation}
where $\omega\in \R$ and  $\psi_{k}$ are real-valued functions decaying to zero at infinity, which by Lemma \ref{H34impGC}
satisfy the following semilinear elliptic system
\begin{equation}\label{systemelip}
\displaystyle -\gamma_{k}\Delta \psi_{k}+\left(\frac{\sigma_{k}\alpha_{k}}{2}\omega+\beta_{k}\right) \psi_{k}=f_{k}(\psib),\qquad k=1,\ldots,l.
\end{equation}
The action functional associated to \eqref{systemelip} is  
\begin{equation}\label{FunctionalI}
I(\boldsymbol{\psi})=\frac{1}{2}\left[\sum_{k=1}^{l}\gamma_{k}\|\nabla \psi_{k}\|_{L^2}^{2}+\sum_{k=1}^{l}\left(\frac{\sigma_{k}\alpha_{k}}{2}\omega+\beta_{k}\right)\| \psi_{k}\|_{L^2}^{2}\right]
-\int F(\boldsymbol{\psi})\;dx.
\end{equation}

A {\it ground state solution} for \eqref{systemelip} is a nontrivial solution  that is a  minimum of $I$ among all solutions of \eqref{systemelip}. Before proceeding, it is convenient to introduce   the functionals
\begin{equation}\label{functionalQ}
\mathcal{Q}(\boldsymbol{\psi})=\sum_{k=1}^{l}\left(\frac{\sigma_{k}\alpha_{k}}{2}\omega+\beta_{k}\right)\| \psi_{k}\|_{L^2}^{2},
\end{equation}
\begin{equation}\label{funclKP6}
K(\boldsymbol{\psi})=\sum_{k=1}^{l}\gamma_{k}\|\nabla \psi_{k}\|_{L^2}^{2},\qquad  P(\boldsymbol{\psi})=\int F(\boldsymbol{\psi})\;dx.
\end{equation}
and the set
\begin{equation*}
     \mathcal{P}:=\{\psib\in \mathbf{H}^{1};\, P(\psib)>0\}.
\end{equation*}

These functionals satisfy some useful identities. 
\begin{lem}
\label{identitiesfunctionals}
Assume $1\leq n\leq 5$ and let $\boldsymbol{\psi}$ be a (weak) solution of \eqref{systemelip}. Then,
\begin{equation}
P(\boldsymbol{\psi})=2I(\boldsymbol{\psi}),\label{b}\\
\end{equation}
\begin{equation}
K(\boldsymbol{\psi})=nI(\boldsymbol{\psi}),\label{d}\\
\end{equation}
\begin{equation}
\mathcal{Q}(\boldsymbol{\psi})=(6-n)I(\boldsymbol{\psi}).\label{e}
\end{equation}
\end{lem}
\begin{proof}
See \cite[Lemma 4.5]{NoPa2}.
\end{proof}

Under our assumptions, ground states for \eqref{systemelip} do exist if the coefficients $\frac{\sigma_{k}\alpha_{k}}{2}\omega+\beta_{k}$ are positive, $k=1,\ldots,l$ and $1\leq n\leq5$.  Thus, if we denote by $\mathcal{G}_n(\omega,\boldsymbol{\beta})$  the set of ground state solutions of \eqref{systemelip}, we have that $\mathcal{G}_n(\omega,\boldsymbol{\beta}) \neq \emptyset$ if $1\leq n\leq 5$. Moreover,   the Gagliardo-Nirenberg-type inequality
\begin{equation}\label{GNI}
    P(\ub)\leq C_n^{opt}\mathcal{Q}(\ub)^{\frac{6-n}{4}}K(\ub)^{\frac{n}{4}}, \qquad \ub\in \mathcal{P},
\end{equation}
holds with the optimal constant   $C_n^{opt}$ given by
\begin{equation}\label{bestCn}
    C_n^{opt}:=\frac{2(6-n)^{\frac{n-4}{4}}}{n^{\frac{n}{4}}}\frac{1}{\mathcal{Q}(\psib)^{\frac{1}{2}}},
\end{equation}
where $\psib \in \mathcal{G}_n(\omega,\boldsymbol{\beta})$, $1\leq n\leq5$ (see \cite[Corollary 4.12]{NoPa2}).

We summarize our global well-posedness results in the following theorem.

\begin{teore}\label{thm:globalwellposH1}
 Assume that \textnormal{\ref{H1}-\ref{H8}} hold and let $\psib$ be a ground state solution of \eqref{systemelip} in $\mathcal{G}_n(1,\boldsymbol{0})$.
	\begin{enumerate}
		\item[(i)] If $1 \leq n\leq 3$ then for any $\mathbf{u}_0\in \mathbf{H}^{1}$,  system \eqref{system1} has a unique  solution $\mathbf{u}\in \mathbf{Y}(\R)$.
		\item[(ii)] Assume $n=4$. Then for any $\mathbf{u}_0\in \mathbf{H}^{1}$ satisfying
			\begin{equation}\label{L2GSCond}
		Q(\mathbf{u}_0)<Q(\psib),	
		\end{equation}
		system \eqref{system1} has a unique  solution $\mathbf{u}\in \mathbf{Y}(\R)$.
		\item[(iii)] Assume $n=5$.  Suppose that $\mathbf{u}_0\in \mathbf{H}^{1}$ satisfies  
		\begin{equation}\label{conditionsharp1}
		Q(\mathbf{u}_0)E(\mathbf{u}_0)<Q(\psib)\mathcal{E}(\psib),
		\end{equation}
		and
			\begin{equation}\label{conditionsharp2}
		Q(\mathbf{u}_0)K(\mathbf{u}_0)<Q(\psib)K(\psib),
		\end{equation}
		where $\mathcal{E}$ is the energy defined in \eqref{Energy} with $\beta_k=0$, $k=1,\ldots,l$. Then system \eqref{system1} has a unique  solution $\mathbf{u}\in \mathbf{Y}(\R)$.   
	\end{enumerate}
\end{teore}
\begin{proof}
	See \cite[Theorems 3.16 and 5.2]{NoPa2}.
\end{proof}

\begin{obs}
\begin{itemize}
\item[(i)] In Section \ref{sec.blowupn56} we will show that under assumption \eqref{conditionsharp1}, condition \eqref{conditionsharp2} is sharp in the sense if inequality has been reversed   and the initial data is radial then the solution must blow up in finite time.
\item[(ii)] In dimension $n=6$, since the existence time in Theorem \ref{locexistH1n=6} depends on the initial data itself, an a priori bound of the local solution is not enough to extend it globally in time.
\end{itemize}
\end{obs}

 \section{Existence of ground states in the $H^{1}$-critical case}\label{sec.exigsn=6}

In this section we are interested in showing the existence of ground state solutions for  \eqref{system1} in the $H^{1}$-critical case. The section can be seen of independent interest since it purely deals with semilinear elliptic equations. In particular assumption \ref{H8} can be dropped here.

For the scalar case, existence of ground-state solutions in the critical case is closely related with  the optimal constant in the critical Sobolev inequality: (see for instance \cite[Theorem 8.3]{Lieb} or \cite[Corollary 1.3]{wang})
\begin{equation}\label{CSI}
\|f\|_{L^{3}}^{2}\leq C\|\nabla f\|_{L^{2}}^{2}, \qquad f\in \dot{H}^{1}(\R^6).
\end{equation}
This was  addressed, for instance, in \cite{talenti} (see also \cite{Lions3}  and \cite[Chapter I, Section 4]{Struwe}) where optimal constant and extremal functions were obtained.

In our case, we first note  from \eqref{e} we must expect non-trivial solutions of \eqref{systemelip} only if 
\begin{equation}\label{betacond}
\frac{\sigma_{k}\alpha_{k}}{2}\omega+\beta_{k}=0,
\end{equation}
which is fulfilled, for instance, if $\omega=0$ and $\boldsymbol{\beta}=\mathbf{0}$.
 Thus, system \eqref{systemelip} and the action functional $I$  reduce to 
\begin{equation}\label{syselip6}
\displaystyle -\gamma_{k}\Delta \psi_{k}=f_{k}(\boldsymbol{\psi})\qquad k=1\ldots,l
\end{equation}
and
\begin{equation}\label{FunclI6}
I(\boldsymbol{\psi})=\frac{1}{2}\sum_{k=1}^{l}\gamma_{k}\|\nabla \psi_{k}\|_{L^2}^{2}
-\int F(\boldsymbol{\psi})\;dx.
\end{equation}

Solutions of \eqref{syselip6} can now be seen as critical points of the action \eqref{FunclI6}. More precisely,  we have the following. 

\begin{defi}
A function $\boldsymbol{\psi}\in \dot{\mathbf{H}}^{1}(\R^{6}) $ is called a solution (weak solution) of \eqref{syselip6} if for any $\mathbf{g}\in \dot{\mathbf{H}}^{1}(\R^{6})$,
\begin{equation}\label{infI}
\displaystyle \gamma_{k} \int \nabla \psi_{k} \cdot \nabla g_{k}\;dx
=\int f_{k}(\boldsymbol{\psi})g_{k}\;dx,\qquad
k=1,\ldots,l.
\end{equation}
\end{defi}

Among all solution of \eqref{syselip6} we single out the ones that minimizes $I$.

\begin{defi}\label{defgroundstate} 
 A solution $\psib\in \dot{\mathbf{H}}^{1}(\R^{6})$ is called a ground state of \eqref{syselip6} if 
\begin{equation*}
I(\psib)=\inf\left\{I(\boldsymbol{\phi}); \boldsymbol{\phi}\in \mathcal{C}\right\},
\end{equation*}
where  $\mathcal{C}$  denotes the set of all non-trivial solutions of \eqref{syselip6}.
We denote by $\mathcal{G}_{6}$ the set of all ground states of \eqref{syselip6}.
\end{defi}

Let us start by observing if $\psib$ is a non-trivial solution of \eqref{syselip6} then the functional $P$ must be positive at $\psib$.

\begin{lem}\label{lemma4.4} Define
$ \D:=\{\psib\in \dot{\mathbf{H}}^{1}(\R^6);\, P(\psib)>0\}$.
Then, $
\mathcal{C}\subset \D
$. 
\end{lem}
\begin{proof}
Let $\psib\in \mathcal{C}$. By taking $\mathbf{g}=\psib$ in \eqref{infI} and using  Lemma \ref{estdifF}-(ii),
\begin{equation}\label{KPrel}
3P(\psib)=K(\psib),
\end{equation}
from which we deduce the desired.
\end{proof}

 It is convenient to introduce the following functionals:
\begin{equation}\label{Jfunc}
    J(\psib):=\frac{K(\psib)^{\frac{3}{2}}}{P(\psib)},\qquad \psib\in \D,
\end{equation}
and
\begin{equation}\label{Efunc}
    \E(\psib):=K(\psib)-2P(\psib)\qquad \psib\in \Dz.
\end{equation}

\begin{obs}\label{relatEIJ}
Let  $\psib$ be  a non-trivial solution of  \eqref{syselip6}. Then, clearly
\begin{equation*}
    \E(\psib)=2I(\psib)
\end{equation*}
and using \eqref{KPrel},
\begin{equation*}
     J(\psib)=\frac{6^{\frac{3}{2}}}{2}I(\psib)^{\frac{1}{2}}.
\end{equation*}
In particular,  a non-trivial solution of \eqref{syselip6}  is a ground state  if and only if its has least energy among all non-trivial solutions of \eqref{syselip6} if only if it minimizes $J$.
\end{obs}

With this in mind, one of the main results of this paper reads as follows.

\begin{teore} \label{thm:GESn6}
There exists a ground state solution $\psib_{0}$ for  system \eqref{syselip6}, that is, $\mathcal{G}_{6}$ is not empty. 
\end{teore}

In order to prove Theorem \ref{thm:GESn6} we shall use the concentration-compactness method to obtain a solution to a constrained minimization problem deduced from a general critical Sobolev-type inequality, which turns out to be a ground state.

 \subsection{General critical Sobolev-type inequality}

  From now on we assume that all components of the vector  $\ub$ are  real-valued functions. 
 Hence, using Lemma \ref{estdifF} we obtain the following general critical Sobolev-type inequality:
  \begin{equation}\label{GCSI6}
     P(\ub)\leq CK(\mathbf{u})^{\frac{3}{2}},\qquad \forall  \mathbf{u}\in \D.
 \end{equation}
 In particular, this shows that functional $J$ is bounded from below by a positive constant. Then, the infimum of $J$ on $\D$ is positive and the best constant we can place in \eqref{GCSI6} is given by
\begin{equation}\label{C6}
    C_{6}^{-1}:=\inf\left\{J(\ub);\;\mathbf{u}\in \D \right\}.
\end{equation}
The subscript in the definition of $C_6$ is motivated by the dimension $n=6$.

We will  prove that the infimum \eqref{C6} is attained. To this end, we consider the following normalized version
\begin{equation}\label{S}
    S:=\inf\left\{K(\mathbf{u});\;\mathbf{u}\in \D,\; P(\ub)=1\right\}.
\end{equation}
Minimization problems as \eqref{S} was studied by Lions in \cite[page 166, equation (30)]{Lions3}. However, since we are not assuming that $F$ is strictly positive outside the origin his approach need to be slightly modified. This is why our minimization problem \eqref{S} is posed on $\D$ and not on $\Dz$.

A \textit{minimizing sequence} for \eqref{C6} is a sequence $(\ub_{m})$ in $\D$ such that $J(\ub_m)\to C_6^{-1}$. In the same way, a minimizing sequence for \eqref{S} is a sequence $(\ub_{m})$ in $\D$ such that $P(\ub_{m})=1$ and $K(\ub_m)\to S$. Since  $K(\!\big\bracevert\!\! \mathbf{u}\!\!\big\bracevert\!)\leq K(\ub)$, assumption \textnormal{\ref{H6}} implies that $J(\!\big\bracevert\!\! \mathbf{u}\!\!\big\bracevert\!)\leq J(\ub)$. Thus, if $(\ub_{m})$ is a minimizing sequence of \eqref{C6} (or \eqref{S}) so is  $(\!\big\bracevert\!\! \mathbf{u}_m\!\!\big\bracevert\!)$. In particular, without loss of generality, we can  (and \textit{will}) assume that minimizing sequences are always non-negative.

\begin{obs}\label{KPinvarice6}
Since the functionals $K$ and $P$ are homogeneous of degree 2 and 3, respectively, we have
\begin{enumerate}
    \item[(i)] $C_{6}=S^{-\frac{3}{2}}$, which means that \eqref{GCSI6} becomes
     \begin{equation}\label{GCSIS}
     P(\ub)\leq S^{-\frac{3}{2}}K(\mathbf{u})^{\frac{3}{2}},\qquad \forall  \mathbf{u}\in \D. 
 \end{equation}
    Moreover,   if  $\mathbf{v}$  is a minimizer for \eqref{S} it also is a minimizer for \eqref{C6}. In fact,
 \begin{equation*}
     J(\vb)=\frac{K(\vb)^{\frac{3}{2}}}{P(\vb)}=K(\vb)^{\frac{3}{2}}=S^{\frac{3}{2}}=C_{6}^{-1}. 
 \end{equation*}
    \item[(ii)] The functionals $K$ and $P$ are invariant under the transformation
    \begin{equation}\label{scaling6}
\ub\mapsto \vb^{R,y}(x)=
 R^{-2}\ub\left(R^{-1}(x-y)\right), 
 \end{equation}
 where $R>0$ and $y\in \R^{6}$. In particular, if $(\ub_{m})$ is a minimizing sequence for \eqref{C6} (or \eqref{S}), so is the sequence  $(\mathbf{v}_{m})$ with $\vb_m(x)=
 R^{-2}\ub_m\left(R^{-1}(x-y)\right)$. 
\end{enumerate}
\end{obs}

 \subsection{Concentration Compactness principle}

To obtain that \eqref{S} has a minimizer we will use the concentration-compactness method. The first result in this direction is based on  \cite[Lemma I.1]{lions1984concentration}. 

 \begin{lem}[Concentration-Compactness I]\label{CCLI}
  Suppose  that $(\nu_{m})$ is a sequence in $\mathcal{M}_{+}^{1}(\R^{n})$. Then, there is a subsequence, still denoted by $(\nu_{m})$, such that one of the following three conditions hold:
  \begin{enumerate}
  
  \item[(i)] (Vanishing) For all $R>0$ there holds 
      \begin{equation*}
          \lim_{m\to \infty}\left(\sup_{x\in \R^{n}}\nu_{m}(B(x,R))\right)=0.
      \end{equation*}
      
      \item[(ii)] (Dichotomy) There exists a number $\lambda$, $0<\lambda<1$, such that for any $\epsilon>0$ there exist a number $R>0$ and a sequence $(x_{m})$ with the following property: given $R'>R$
      \begin{equation*}
          \nu_{m}(B(x_{m},R))\geq \lambda-\epsilon,
      \end{equation*}
      \begin{equation*}
         \nu_{m}(\R^{n}\setminus B(x_{m},R'))\geq 1-\lambda-\epsilon,
      \end{equation*}
      for $m$ sufficiently large. 
       \item[(iii)] (Compactness) There exists a sequence $(x_{m})\subset \R^{n}$ such that for any $\epsilon>0$ there is a radius $R>0$ with the property that 
      \begin{equation*}
     \nu_{m}(B(x_{m},R))\geq 1-\epsilon,
      \end{equation*}
      for all $m$.
     
  \end{enumerate}
 \end{lem}
 
 \begin{proof}
 See for instance \cite[Chapter I, Lemma 4.3]{Struwe} and \cite[Lemma 23]{Flucher}.
 \end{proof}
 
  The next lemma is inspired by the concentration-compactness principle in the limiting case (see \cite{Lions3}).  For its proof we follow closely the ideas presented in \cite[Theorem 1.4.2]{Evans1992measure} (see also \cite[Lemma 4.8]{Struwe}).

 \begin{lem}[Concentration-compactness II]\label{thm:CCLII} 
 Let $(\mathbf{u}_{m})\subset \Dz $ be any sequence such that $\ub_{m}\geq \mathbf{0}$ and
  \begin{equation}\label{weakconv}
  \begin{cases}
  \mathbf{u}_{m} \rightharpoonup \mathbf{u}, \qquad \qquad \qquad\qquad\quad\quad \mathrm{ in }\quad\Dz,\\
 {\displaystyle \mu_{m}:=\sum_{k=1}^{l}\gamma_{k}|\nabla u_{km}|^{2}\;dx} \overset{\ast}{\rightharpoonup} \mu, \quad \mathrm{ in}\quad \mathcal{M}^{b}_{+}(\R^{6}),\\ 
  \nu_{m}:=F(\mathbf{u}_{m})\;dx \overset{\ast}{\rightharpoonup} \nu, \qquad \qquad \quad \mathrm{ in}\quad \mathcal{M}^{b}_{+}(\R^{6}).
  \end{cases}
  \end{equation}
  Then,
  \begin{enumerate}
      \item[(i)] There exist an at most countable set $J$, a family of distinct points $\{x_{j}\in \R^{6}:j\in J\}$,  and a family  of non-negative numbers $\{\nu_{j}:j\in J\}$ such that
      \begin{equation}\label{nudesc}
          \nu=F(\mathbf{u})\; dx+\sum_{j\in J}\nu_{j}\delta_{x_{j}}.
      \end{equation}
      \item[(ii)] In addition, we have
      \begin{equation}\label{mudesc}
          \mu\geq \sum_{k=1}^{l}\gamma_{k}|\nabla u_{k}|^{2}\;dx+\sum_{j\in J}\mu_{j}\delta_{x_{j}},
      \end{equation}
      for some family $\{\mu_{j}:j\in J\}$, $\mu_{j}>0$, such that
      \begin{equation}\label{Sobvj}
          \nu_{j}\leq S^{-\frac{3}{2}}\mu_{j}^{\frac{3}{2}}, \qquad \forall j \in J.
      \end{equation}
      In particular, $\sum_{j\in J}\nu_{j}^{\frac{2}{3}}<\infty$. 
  \end{enumerate}
 \end{lem} 
  
  \begin{obs}\label{umnoneg}
   Since $\ub_{m}\geq \mathbf{0}$, Lemma \ref{estdifF} (iii) implies that $F(\ub_{m})\geq 0$. Hence, $\nu_{m}$ is indeed a positive measure. Moreover, the weak convergence $\ub_m\rightharpoonup\ub$ implies that, up to a subsequence, $\ub_{m}\to \ub$ a.e. in $\R^{6}$ (see for instance   \cite[Corollary 8.7]{Lieb}). As a consequence,  $\ub\geq \mathbf{0}$.  
  \end{obs}
  \begin{proof}[Proof of Lemma \ref{thm:CCLII}] We divide the proof into the cases $\ub=0$ and $\ub\neq0$. 
  \vskip.2cm

\noindent {\bf Step 1.}  Assume first that $\mathbf{u}=\mathbf{0}$. 

Let $\xi \in \mathcal{C}_{c}^{\infty}(\R^{6})$. From the vague convergence of $(\nu_{m})$ in \eqref{weakconv} and assumption \ref{H5} we have
  \begin{equation}\label{lim1}
    \begin{split}
        \int|\xi|^{3}\;d\nu
        %&=\lim_{m\to \infty}\int|\xi|^{3}\;d\nu_{m}\\
        =\lim_{m\to \infty}\int|\xi|^{3}F(\mathbf{u}_{m})\;dx=\lim_{m\to \infty}\int F(|\xi|\mathbf{u}_{m})\;dx\leq S^{-\frac{3}{2}}\liminf_{m\to \infty}K(\xi\mathbf{u}_{m})^{\frac{3}{2}},
    \end{split}
  \end{equation}
  where we have used the critical Sobolev-type inequality \eqref{GCSIS} in the last inequality. Since $\mathbf{u}_{m}\rightharpoonup  \mathbf{0}$ in $\Dz$ we know that (see \cite[Theorem 8.6]{Lieb}),  for any $A\subset \R^{6}$ with finite measure and $k=1,\ldots,l$ we have 
  \begin{equation}\label{strogconv}
      \chi_{A}u_{km}\to 0, \quad \mathrm{ strongly\;\;\; in}\quad L^{2}(\R^{6}).
  \end{equation}
 Thus, using the triangular inequality and taking $A$ as $\mathrm{supp}( |\nabla \xi|)$ in \eqref{strogconv}   we get
\begin{equation*}
    \begin{split}
         \left|\left(\sum_{k=1}^{l}\gamma_{k}\|\nabla[\xi u_{km}]\|^{2}_{L^{2}}\; \right)^{\frac{1}{2}}-\left(\sum_{k=1}^{l}\gamma_{k}\|\xi\nabla u_{km}\|^{2}_{L^{2}}\right)^{\frac{1}{2}}\right|
        &\leq\left(\sum_{k=1}^{l}\gamma_{k}\|\nabla[\xi u_{km}]-\xi\nabla u_{km}\|^{2}_{L^{2}}\right)^{\frac{1}{2}} \\
        &=\left(\sum_{k=1}^{l}\gamma_{k}\|u_{km}\nabla\xi \|^{2}_{L^{2}}\right)^{\frac{1}{2}}\\
        &\leq C\left(\sum_{k=1}^{l}\int|\chi_{A} u_{km}|^{2}\;dx\right)^{\frac{1}{2}}\\
        &\to 0,\qquad \mathrm{as} \quad m\to \infty.
    \end{split}
\end{equation*}
Combining this with  the vague convergence of $(\mu_{m})$  we obtain
    \begin{equation*}
    \begin{split}
       \liminf_{m\to \infty}K(\xi\ub_{m})^{\frac{3}{2}}
        &=\liminf_{m\to \infty}\left(\int|\xi|^{2}\sum_{k=1}^{l}\gamma_{k}|\nabla u_{km}|^{2}\; dx\right)^{\frac{3}{2}}\\
        &=\liminf_{m\to \infty}\left(\int|\xi|^{2}\;d\mu_{m}\right)^{\frac{3}{2}}\\
        &=\left(\int|\xi|^{2}\;d\mu\right)^{\frac{3}{2}}.
    \end{split}
    \end{equation*}  
 Therefore, from \eqref{lim1}  we deduce that
 \begin{equation}\label{ineqxi}
     \int|\xi|^{3}\;d\nu\leq S^{-\frac{3}{2}}\left(\int|\xi|^{2}\;d\mu\right)^{\frac{3}{2}},\qquad  \xi \in \mathcal{C}^{\infty}_{c}(\R^{6}). 
 \end{equation}
 We claim that inequality \eqref{ineqxi} actually implies that
 \begin{equation}\label{ineqbor}
     \nu(E)\leq  S^{-\frac{3}{2}}\mu(E)^{\frac{3}{2}}, \qquad \mbox{for any}\; E\in \mathcal{B}(\R^{6}).
\end{equation}
In fact,  since $\nu$ and $\mu$ are Radon measures, they are inner regular on open sets and outer regular on Borel sets, respectively. Let $U\subset \R^{6}$ be an open set and take  any compact set $A$, with $A\subset U $. By  $\mathcal{C}^{\infty}$ Urysohn's  lemma (see for instance \cite[Lemma 8.18]{Folland}) there exists $f\in  \mathcal{C}^{\infty}_{c}(\R^{6})$ such that $0\leq f\leq 1$, $f=1$ on $A$ and $\mathrm{supp}(f)\subset U$. Then, from \eqref{ineqxi},
   \begin{equation*}
          \nu(A)=\!\int_{A} f^{3}\; d\nu\leq \int f^{3}\; d\nu\leq S^{-\frac{3}{2}}\left(\int f^{2}\;d\mu\right)^{\frac{3}{2}}\!\leq \! S^{-\frac{3}{2}}\left(\int_{\mathrm{supp}(f)}f^{2}\;d\mu\right)^{\frac{3}{2}}\!\leq\! S^{-\frac{3}{2}}\left(\int_{{U}}\;d\mu\right)^{\frac{3}{2}}.
  \end{equation*}
  Hence, $\nu(A)\leq  S^{-\frac{3}{2}}\mu(U)^{\frac{3}{2}}$ for all $A\subset U$, $A$ compact. Thus, from the inner regularity of the measure $\nu$ we conclude that
\begin{equation}\label{ineqopen}
     \nu(U)\leq  S^{-\frac{3}{2}}\mu(U)^{\frac{3}{2}}, \qquad \mbox{for any}\; U\subset \R^{6}, \quad U\quad \mathrm{open}.
\end{equation}
Now, consider any $E\in \mathcal{B}(\R^{6})$ and  let $U$ be an open subset with $E\subset U$. Then, from \eqref{ineqopen} we have $\nu(E)\leq \nu(U)\leq S^{-\frac{3}{2}}\mu(U)^{\frac{3}{2}}$. It follows from the outer regularity of the measure $\mu$  that $\nu(E)\leq  S^{-\frac{3}{2}}\mu(E)^{\frac{3}{2}}$.

Next let  $D$ be the set of atoms of the measure $\mu$, i.e., $D=\{x\in \R^{6}: \mu(\{x\})>0\}$. Note that $D=\bigcup_{k=1}^{\infty}D_{k}$ with $D_{k}=\left\{x\in \R^{6}: \mu(\{x\})>\frac{1}{k}\right\}$. Since $\mu$ is finite it follows that $D_{k}=\left\{x\in \R^{6}: \mu(\{x\})>\frac{1}{k}\right\}$ is finite for all $k$, from which we deduce that  $D$ is at most countable.  Thus, we can write $D=\{x_{j}:j\in J\}$, where $J$ is a countable subset of $\mathbb{N}$.

Define $\mu_{j}:=\mu(\{x_{j}\})$, $j\in J$. For any $E\in \mathcal{B}(\R^{6})$ we have
\begin{equation}\label{mupos}
    \sum_{j\in J}\mu_{j}\delta_{x_{j}}(E)=\sum_{\substack{j\in J\\x_{j}\in E}}\mu_{j}=\sum_{\substack{j\in J\\x_{j}\in E}}\mu(\{x_{j}\})\leq \mu(E). 
\end{equation}
which proves \eqref{mudesc} in the case $\mathbf{u}=\mathbf{0}$. 

Now, we will prove that \eqref{nudesc} also holds. From \eqref{ineqbor} we have $\nu\ll \mu$, then by the Radon-Nikodym Theorem (see for instance   \cite[Section 1.6]{EvansGas})  there exists a non-negative function  $h\in L^{1}(\R^{6},\mu)$ such that
\begin{equation}\label{RNTh}
    \nu(E)=\int_{E}h(x)d\mu(x), \qquad \mbox{for any}\; E \in \mathcal{B}(\R^{6}).
\end{equation}
In addition, $h$ satisfies
\begin{equation}\label{limh}
     h(x)=\lim_{r\to 0}\frac{\nu(B(x,r))}{\mu(B(x,r))}, \qquad  \mu \quad \mathrm{a.e.}\quad  x\in \R^{6}. 
\end{equation}
Combining \eqref{limh} and  \eqref{ineqbor} we get $0\leq h(x)\leq S^{-\frac{3}{2}}\mu(\{x\})^{\frac{1}{2}}$. This shows that $ h(x)=0$, $\mu$ a.e.  on $\R^{6}\setminus D$. In particular, $h$ assumes countable many values and,  consequently, the integral in \eqref{RNTh} can be represented (see for instance \cite[Example 2.5.8]{bogachev2007measure}) by
\begin{equation}\label{repreintE}
    \int_{E}h(x)d\mu(x)=\sum_{\substack{j\in J\\x_{j}\in E}}h(x_{j})\mu(\{x_{j}\}).
\end{equation}
Define $\nu_{j}:=\nu(\{x_{j}\})$, $j\in J$. We see from \eqref{RNTh} and \eqref{repreintE} that in fact $\nu_{j}=h(x_{j})\mu_{j}$, for all $j\in J$. Therefore, fo any $E \in \mathcal{B}(\R^{6})$, 
 \begin{equation*}
      \nu(E)=\sum_{\substack{j\in J\\x_{j}\in E}}h(x_{j})\mu(\{x_{j}\})=\sum_{\substack{j\in J\\x_{j}\in E}}\nu_{j}=\sum_{j\in J}\nu_{j}\delta_{x_{j}}(E),
  \end{equation*}
 which is \eqref{nudesc}
 with $\mathbf{u}=\mathbf{0}$. 
 
 Finally, inequality \eqref{Sobvj} follows immediately from the definitions of $\mu_{j}$, $\nu_{j}$ and  \eqref{ineqbor}. Note also that by taking $E=\R^n$ in \eqref{mupos} we deduce that $\sum_{j\in J}\mu_{j}$ is convergent. Hence, the convergence of the series $\sum_{j\in J}\nu_{j}^{\frac{2}{3}}$ follows from \eqref{ineqbor}.
\vskip.2cm

\noindent {\bf Step 2.}  Assume now $\ub\neq \mathbf{0}$. First note that  Lemma \ref{estdifF} (iii) implies $F(\ub)\geq0$, so $F(\ub)\;dx$ defines a positive measure.

\noindent {\bf Claim.}
 The measures 
\begin{equation}\label{difmunu}
    \mu- \sum_{k=1}^{l}\gamma_{k}|\nabla u_{k}|^{2}\;dx \qquad\mathrm{and}\qquad \nu-F(\ub)\;dx,
\end{equation}
are non-negative. 

To prove this, define
$\vb_{m}=\ub_{m}-\ub$ 
and consider  the sequences of measures
\begin{equation*}
    \tilde{\mu}_{m}:=\sum_{k=1}^{l}\gamma_{k}|\nabla v_{km}|^{2}\;dx\qquad\mathrm{and}\qquad \tilde{\nu}_{m}:= F(\!\!\big\bracevert\!\!\vb_{m}\!\!\big\bracevert\!\!)\;dx.
\end{equation*}
Recall that $\big\bracevert\!\!\vb_{m}\!\!\big\bracevert$ denotes the vector $(|v_{1m}|, \ldots, |v_{km}|)$.
Since $\mathbf{v}_{m}\rightharpoonup \mathbf{0}$ in  $\Dz$,
 the sequence $(K(\mathbf{v}_{m}))$ is uniformly bounded. In view of
\begin{equation*}\label{vagemutil}
    \left|\int f \;  d\tilde{\mu}_{m}\right|\leq \|f\|_{L^{\infty}}K(\mathbf{v}_{m}), \qquad f\in \mathcal{C}_c(\R^6),
\end{equation*}
it follows that  $(\tilde{\mu}_{m})$ is a vaguely bounded sequence in $\mathcal{M}_{+}^{b}(\R^{6})$. 
Hence, by Lemma \ref{vagbndconv} there exists a subsequence, still denoted by $(\tilde{\mu}_{m})$, and $\tilde{\mu}\in \mathcal{M}_{+}^b(\R^{6})$ such that 
\begin{equation}\label{mutilvague}
   \tilde{\mu}_{m} \overset{\ast}{\rightharpoonup} \tilde{\mu}, \qquad \mathrm{in}\qquad\mathcal{M}^{b}_{+}(\R^{6}). 
\end{equation}
We claim that 
\begin{equation}\label{mutilvague1}
\mu_{m}\overset{\ast}{\rightharpoonup}\tilde{\mu}+ \sum_{k=1}^{l}\gamma_{k}|\nabla u_{k}|^{2}\;dx \qquad \mathrm{in}\qquad\mathcal{M}^{b}_{+}(\R^{6}).
\end{equation}
 If this is the case, since the vague limit is unique, we have
\begin{equation*}
    \mu=\tilde{\mu}+ \sum_{k=1}^{l}\gamma_{k}|\nabla u_{k}|^{2}\;dx.
\end{equation*}
Since all the measures involved are finite, it follows  that the first difference in \eqref{difmunu} is a non-negative measure.

 Let us prove \eqref{mutilvague1}. Taking into account that $\partial_{x_{i}}v_{km}\rightharpoonup0$ in $L^{2}(\R^{6})$ and $f\partial_{x_{i}}u_{k}\in L^2(\R^6)$, for any $f\in \mathcal{C}_c(\R^6)$, we deduce
\begin{equation}\label{convnabla}
    \lim_{m\to \infty}\int f \nabla v_{km}\cdot\nabla u_{k}\;dx=0, \qquad k=1,\ldots,l.
\end{equation}
 Thus, for any  $f\in \mathcal{C}_{c}(\R^{6})$ we get
\begin{equation*}
    \begin{split}
        0&\leq \left|\int f\;d\mu_{m}-\int f\left[d\tilde{\mu}+\sum_{k=1}^{l}\gamma_{k}|\nabla u_{k}|^{2}\;dx\right]\right|\\
        &=\left|\int f \sum_{k=1}^{l}\gamma_{k}|\nabla u_{km}|^{2}\;dx-\int f\left[d\tilde{\mu}+\sum_{k=1}^{l}\gamma_{k}|\nabla u_{k}|^{2}\;dx\right]\right|\\
        &=\left|\int f\sum_{k=1}^{l}\gamma_{k}\left(|\nabla v_{km}|^{2}+2\nabla v_{km}\cdot \nabla u_{k}+|\nabla u_k|^2\right)dx -\int fd\tilde{\mu} - \int f\sum_{k=1}^{l}\gamma_{k}|\nabla u_{k}|^{2}\;dx \right|\\
        &\leq \left|\int f\;d\tilde{\mu}_{m}-\int f\;d\tilde{\mu}\right|+2\sum_{k=1}^{l}\gamma_{k}\left|\int f \nabla v_{km}\cdot \nabla u_{k}\;dx\right|
    \end{split}
\end{equation*}
The first term goes to zero by the vague convergence in \eqref{mutilvague}, the second one goes to zero by \eqref{convnabla}. This establishes \eqref{mutilvague1}.

Next, we are going to prove that  $(\tilde{\nu}_{m})$ is also  vaguely bounded in $\mathcal{M}_{+}^{b}(\R^{6})$. As we point out before, $(K(\vb_{m}))$ is uniformly bounded.  Hence, from the critical Sobolev inequality \eqref{CSI} $(\vb_{m})$ is uniformly bounded in $\mathbf{L}^{3}(\R^{6})$. 
 Thus, for any $f\in \mathcal{C}_{c}(\R^{6})$,
 \begin{equation*}
    \left|\int f\; d\tilde{\nu}_{m}\right|\leq \|f\|_{L^{\infty}}\int F(\!\!\big\bracevert\!\!\vb_{m}\!\!\big\bracevert\!\!)\;dx \leq C\int \sum_{k=1}^{l}|v_{km}|^{3}\;dx\leq M_{2},
\end{equation*}
 for some constant $M_{2}$.  From Lemma \ref{vagbndconv} again  there exist a subsequence, still denoted by $(\tilde{\nu}_{m})$, and a measure $\tilde{\nu}\in \mathcal{M}_{+}^{b}(\R^{6})$ such that 

\begin{equation}\label{nutilvague}
   \tilde{\nu}_{m} \overset{\ast}{\rightharpoonup} \tilde{\nu}, \qquad \mathrm{in}\qquad\mathcal{M}^{b}_{+}(\R^{6}). 
\end{equation}

We claim that
\begin{equation}\label{nutilvague1}
\nu_{m}\overset{\ast}{\rightharpoonup}\tilde{\nu}+ F(\mathbf{u})\;dx \qquad \mathrm{in}\qquad\mathcal{M}^{b}_{+}(\R^{6}),
\end{equation} 
which implies that $\nu=\tilde{\nu}+ F(\mathbf{u})\;dx$ and, therefore, the measure $\nu- F(\mathbf{u})\;dx$ is non-negative. 

We prove \eqref{nutilvague1}  by using the generalized version of Brezis-Lieb's lemma stated in Lemma \ref{BLLG} with $F(\!\!\big\bracevert\!\!x\!\!\big\bracevert\!\!)$ instead of $F(x)$. Indeed, first we may assume that $\mathbf{v}_{m}\to \mathbf{0}$ a.e. on $\R^{6}$ (see Remark \ref{umnoneg}). Now, since  $\mathbf{u}\in\mathbf{L}^{3}(\R^{6})$, Lemma \ref{estdifF}  implies that $F(\!\!\big\bracevert\!\!\mathbf{u} \!\!\big\bracevert\!\!)\in {L}^{1}(\R^{6})$. Moreover,  the sequence $(\vb_{m})$ is uniformly bounded in $\mathbf{L}^{3}(\R^{6})$. Hence, if  $\varphi$ and $\psi_{\epsilon}$ are the functions defined in \eqref{FBL} we have
\begin{equation*}
    \int\varphi(\mathbf{v}_{m})\;dx\leq M,\qquad \mathrm{and}\qquad \int\psi_{\epsilon}(\mathbf{u})\;dx <\infty,
\end{equation*}
for some constant $M$ independent of $\epsilon>0$ and $m$. Lemma \ref{BLLG} then yields
\begin{equation}\label{BLforF}
   \lim_{m\to \infty}\int |F(\!\!\big\bracevert\!\!\ub_{m}\!\!\big\bracevert\!\!)- F(\!\!\big\bracevert\!\!\vb_{m}\!\!\big\bracevert\!\!) - F(\!\!\big\bracevert\!\!\ub\!\!\big\bracevert\!\!)|\;dx =0,
\end{equation} 
where we have used that $\big\bracevert\!\!\ub_{m}\!\!\big\bracevert =\ub_m$ and $\big\bracevert\!\!\ub\!\!\big\bracevert=\ub$.
Thus,  for any $f\in \mathcal{C}_{c}(\R^{6})$,
\begin{equation*}
    \begin{split}
        0&\leq\left|\int f\;d\nu_{m}-\int f\left[d\tilde{\nu}+F(\ub)\;dx\right]\right|\\ 
        &=\left|\int fF(\ub_{m})\;dx-\int fF(\!\!\big\bracevert\!\!\vb_{m}\!\!\big\bracevert\!\!)\;dx+\int fF(\!\!\big\bracevert\!\!\vb_{m}\!\!\big\bracevert\!\!)\;dx-\int f\left[d\tilde{\nu}+F(\ub)\;dx\right]\right|\\
        &\leq \|f\|_{L^{\infty}}\int |F(\!\!\big\bracevert\!\!\ub_{m}\!\!\big\bracevert\!\!)- F(\!\!\big\bracevert\!\!\vb_{m}\!\!\big\bracevert\!\!) - F(\!\!\big\bracevert\!\!\ub\!\!\big\bracevert\!\!)|\;dx+\left|\int f\;d\tilde{\nu}_{m}-\int f\;d\tilde{\nu}\right|.
    \end{split}
\end{equation*}
The first term goes to zero by \eqref{BLforF} and the second one goes to zero by the vague convergence  \eqref{nutilvague}. This proves that the second measure in \eqref{difmunu} is also non-negative and the proof of the claim is completed.

Finally, from the proof of the above claim
\begin{equation*}\label{weakconv2}
  \begin{cases}
{\displaystyle  \sum_{k=1}^{l}\gamma_{k}|\nabla v_{km}|^{2}\;dx \overset{\ast}{\rightharpoonup} \mu-\sum_{k=1}^{l}\gamma_{k}|\nabla u_{k}|^{2}\;dx, \quad \mathrm{ in}\quad \mathcal{M}^{b}_{+}(\R^{6})},\\
  F(\!\!\big\bracevert\!\!\vb_{m}\!\!\big\bracevert\!\!)\;dx
  \overset{\ast}{\rightharpoonup} \nu-F(\ub)\;dx, \qquad \qquad \qquad\qquad
  \mathrm{ in}\quad \mathcal{M}^{b}_{+}(\R^{6}).
  \end{cases}
  \end{equation*}
 So the proof of the lemma is completed if now we apply Step 1. It must be observed that Step 1 holds if we do not have $\ub_m\geq\mathbf{0}$  but replace the sequence $(\nu_m)$ in \eqref{weakconv} by $\nu_m:=F(\!\!\big\bracevert\!\!\ub_{m}\!\!\big\bracevert\!\!)dx$.
  \end{proof}

  The next result is useful to construct a localized  Sobolev-type inequality. 
  
  \begin{lem}\label{teclocGSI} 
   For every $\delta>0$ there is a constant $C(\delta)>0$ with the following property: if $0<r<R$ with $r/R\leq C(\delta)$ and $x\in\R^6$, then there is a cut-off function $\chi_{R}^{r}\in W^1_{\infty}(\R^{6})$ such that $\chi_{R}^{r}=1  $ on  $B(x,r)$,  $\chi_{R}^{r}=0$ outside $B(x,R)$,
   \begin{equation}\label{locas1}
     K(\chi_R^r\ub)\leq \sum_{k=1}^l\gamma_k  \int_{B(x,R)}|\nabla u_k|^{2}\;dy+\delta K(\ub),
   \end{equation}
   and
   \begin{equation}\label{locas2}
   K\big((1-\chi_R^r)\ub\big)\leq \sum_{k=1}^l\gamma_k  \int_{\R^6\setminus B(x,r)}|\nabla u_k|^{2}\;dy+\delta K(\ub),
   \end{equation}
   for any $\ub\in \Dz$.
   \end{lem}
\begin{proof}
This result was essentially proved in  \cite[Lemma 8]{Flucher}.  Without loss of generality assume $x=0$. The function $\chi_R^r$ is given by
\[
\chi_R^r(y)=
\begin{cases}
1, \qquad |y|\leq r,\\
\dfrac{\log(|y|/R)}{\log(r/R)}, \quad r\leq|y|\leq R,\\
0, \qquad |y|\geq R.
\end{cases}
\]
It is easy to see that $\chi_{R}^{r}\in W^1_{\infty}(\R^{6})$ and
\begin{equation}\label{locas3}
\int_{B(0,R)}|\nabla \chi_R^r|^6=\frac{\omega_6}{(\log(R/r))^{5}},
\end{equation}
where $\omega_6$ is the measure of the unit sphere in $\R^6$.

Next observe that Young and H\"older's inequalities, \eqref{CSI}, and \eqref{locas3} imply, for any $\varepsilon>0$,
\[
\begin{split}
\int_{B(0,R)}\left|\nabla [\chi_{R}^{r}u_k]\right|^{2}\;dy&\leq (1+\varepsilon)\int_{B(0,R)}|\chi_R^r|^2|\nabla u_k|^2\;dy+\left(1+\frac{1}{\varepsilon}\right) \int_{B(0,R)}|u_k|^2|\nabla \chi_R^r|^2\;dy\\
& \leq (1+\varepsilon)\int_{B(0,R)}|\chi_R^r|^2|\nabla u_k|^2\;dy+\left(1+\frac{1}{\varepsilon}\right)\!\|u_k\|_{L^3}^2\!\left(\int_{B(0,R)}\!\!\!|\nabla\chi_R^r|^6dy \right)^{\frac{1}{3}}\\
&\leq  (1+\varepsilon)\int_{B(0,R)}|\chi_R^r|^2|\nabla u_k|^2\;dy+\left(1+\frac{1}{\varepsilon}\right)\frac{C\omega_6^{\frac{1}{3}}}{(\log(R/r))^{\frac{5}{3}}}\int_{\R^6}|\nabla u_k|^2\;dy
\end{split}
\]
Multiplying the above expression by $\gamma_k$ and summing up we obtain
$$
 K(\chi_R^r\ub)\leq \sum_{k=1}^l\gamma_k  \int_{B(0,R)}|\nabla u_k|^{2}\;dy+\left[\varepsilon + \left(1+\frac{1}{\varepsilon}\right)\frac{\zeta^2}{(\log(R/r))^{\frac{5}{3}}} \right]K(\ub),
$$
where $\zeta=\sqrt{C}\omega_6^{\frac{1}{6}}$. By taking $\varepsilon=\sqrt{\delta+1}-1$ and
$$
C(\delta):=\exp\left[-\left(\frac{\zeta}{\sqrt{\delta+1}-1}\right)^{\frac{6}{5}}\right],
$$
we see that if $r/R\leq C(\delta)$ then
$$
\varepsilon + \left(1+\frac{1}{\varepsilon}\right)\frac{\zeta^2}{(\log(R/r))^{\frac{5}{3}}}\leq \delta
$$
and \eqref{locas1} follows. For \eqref{locas2} note that
\[
\begin{split}
\int_{\R^6\setminus B(0,r)}\left|\nabla [(1-\chi_{R}^{r})u_k]\right|^{2}\;dy&\leq (1+\varepsilon)\int_{\R^6\setminus B(0,r)}|1-\chi_R^r|^2|\nabla u_k|^2\;dy\\
&\quad+\left(1+\frac{1}{\varepsilon}\right) \int_{\R^6\setminus B(0,r)}|u_k|^2|\nabla(1- \chi_R^r)|^2\;dy.\\
\end{split}
\]
So, since $|\nabla(1- \chi_R^r)|^2=|\nabla\chi_R^r|^2$ and $\chi_R^r=0$ outside $B(0,R)$ we see that \eqref{locas2} follows as in \eqref{locas1}.
\end{proof}

Now we are able to establish the following localized version of the Sobolev inequality. It will be used to rule out dichotomy in the the concentration-compactness lemma below.
  
 \begin{coro}\label{locGSI} Let $\mathbf{u}\in \Dz$ with $\ub \geq \mathbf{0}$. Fix $\delta>0$ and $r/R\leq C(\delta)$ with $C(\delta)$ as in Lemma \ref{teclocGSI}, then
 \begin{equation}\label{FB0r}
     \int_{B(x,r)}F(\mathbf{u})\;dy\leq S^{-\frac{3}{2}}\left[\int_{B(x,R)}\sum_{k=1}^{l}\gamma_{k}|\nabla u_{k}|^{2}\;dy+\delta K(\mathbf{u}) \right]^{\frac{3}{2}},
 \end{equation}
 \begin{equation}\label{FB0R}
     \int_{\R^{6}\setminus B(x,R)}F(\mathbf{u})\;dy\leq S^{-\frac{3}{2}}\left[\int_{\R^{6}\setminus B(x,r)}\sum_{k=1}^{l}\gamma_{k}|\nabla u_{k}|^{2}\;dy+(2\delta+\delta^{2}) K(\mathbf{u}) \right]^{\frac{3}{2}}.
 \end{equation}
  
 \end{coro}

 \begin{proof} Without loss of generality we may assume $x=0$. Note that $\chi_{R}^{r}=1$ on $B(0,r)$ and $\mathrm{supp}(\chi_{R}^{r})=\overline{B(0,R)}$. Then, \eqref{GCSIS} and \eqref{locas1} give
 \begin{equation*}
     \begin{split}
         \int_{B(0,r)}F(\ub)\;dx
         &\leq \int_{\R^{6}}F(\chi_{R}^{r}\ub)\;dx\\
         &\leq S^{-\frac{3}{2}}K(\chi_R^r\ub)^{\frac{3}{2}}\\
         &\leq S^{-\frac{3}{2}}\left[\sum_{k=1}^{l}\int_{B(0,R)}\gamma_k|\nabla u_{k}|^{2}\;dx+\delta K(\ub)\right]^{\frac{3}{2}},
     \end{split}
 \end{equation*}
 which is \eqref{FB0r}. To prove  \eqref{FB0R}  we use the cut-off function $(1-\chi^{r}_{R})\chi^{R_{1}}_{R_{2}}$, with $r<R<R_{1}<R_{2}$ and $R_1/R_2\leq C(\delta)$. Indeed, since $(1-\chi^{r}_{R})\chi^{R_{1}}_{R_{2}}=1$ on $B(0,R_{1})\setminus B(0,R)$ we have
 \begin{equation*}
     \begin{split}
        \int_{B(0,R_{1})\setminus B(0,R)}F(\mathbf{u})\;dx &=
        \int_{B(0,R_{1})\setminus B(0,R)}F\left(\chi^{R_{1}}_{R_{2}}(1-\chi^{r}_{R})\ub\right)\;dx\\
        &\leq  \int_{B(0,R_{1})}F\left(\chi^{R_{1}}_{R_{2}}(1-\chi^{r}_{R})\ub\right)\;dx\\
        & \leq S^{-\frac{3}{2}}\left[\sum_{k=1}^{l}\int_{B(0,R_{2})}\gamma_k|\nabla[ (1-\chi^{r}_{R})u_{k}]|^{2}\;dx+\delta K((1-\chi^{r}_{R})\ub)\right]^{\frac{3}{2}},
     \end{split}
 \end{equation*}
 where we have used \eqref{FB0r} in the last inequality. Since $R_1$ and $R_2$ can be taken arbitrarily large satisfying $R_1/R_2\leq C(\delta)$, the above inequality implies that
 \begin{equation*}
 \begin{split}
 \int_{\R^6\setminus B(0,R)}F(\mathbf{u})\;dx 
 & \leq S^{-\frac{3}{2}}\left[K\big((1-\chi_R^r)\ub\big)+\delta K((1-\chi^{r}_{R})\ub)\right]^{\frac{3}{2}}.
 \end{split}
 \end{equation*}
 Finally, \eqref{locas2} yields
 $$
  \int_{\R^6\setminus B(0,R)}F(\mathbf{u})\;dx \leq S^{-\frac{3}{2}}\left[ \sum_{k=1}^l\gamma_k  \int_{\R^6\setminus B(x,r)}|\nabla u_k|^{2}\;dy+\delta K(\ub) +\delta(1+\delta)K(\ub)\right]^{\frac{3}{2}},
 $$
 which is the desired.
 \end{proof}

 The following result is an adapted version of Lemma 1.7.4 in    \cite{Cazenave}.

 \begin{lem}\label{QmRattain} Let $(\ub_{m}) \subset \mathbf{L}^{3}(\R^{6})$ be such that $\ub_{m}\geq \mathbf{0}$ and
 $ \int F(\ub_{m})\;dx=1,$ for all $m$. Consider
  the concentration function $Q_{m}(R)$ of  $F(\ub_{m})$, i.e., 
 \begin{equation*}
    Q_{m}(R)=\sup_{y\in\R^{6}}\int _{B(y,R)} F(\ub_{m})\;dx, \qquad R>0. 
 \end{equation*}
 Then, for each $m$ there exists $y=y(m,R)\in \R^{6}$ such that 
  \begin{equation*}
      Q_{m}(R)=\int_{B(y,R)} F(\ub_{m})\;dx.
  \end{equation*}
 \end{lem}
  \begin{proof}
  Fix $m\in \N$. Given $R>0$, from the definition of $Q_{m}$ there exists a sequence $(y_{j})\subset \R^{6}$ such that
  \begin{equation*}
      Q_{m}(R)= \lim_{j\to \infty}\int_{B(y_{j},R)} F(\ub_{m})\;dx>0.
  \end{equation*}
  Thus, there is  $j_{0}$ such that if $j\geq j_{0}$ then $\int_{B(y_{j},R)} F(\ub_{m})\;dx\geq \varepsilon,
  $
  where $\varepsilon$ is a positive constant.
  
  We claim that $(y_{j})$ is bounded. Otherwise, there exists a infinite subsequence, still denoted by $(y_{j})$, such that $B(y_{j},R)\cap B(y_{i},R)=\emptyset$, for all $i\neq j$.  Then,
  \begin{equation*}
      1= \int F(\ub_{m})\;dx\geq \sum_{j\geq j_{0}}\int_{ B(y_{j},R)} F(\ub_{m})\;dx=+\infty, 
  \end{equation*}
  which is a contradiction.   Hence, $(y_{j})$ has a convergent subsequence $(y_{j_{s}})$, with limit $y=y(m,R)$. An application of the dominated convergence theorem gives
  \begin{equation*}
      Q_{m}(R)=\lim_{j_s\to \infty}\int_{B(y_{j_s},R)} F(\ub_{m})\;dx=\int_{B(y,R)} F(\ub_{m})\;dx,
  \end{equation*}
  and the proof is completed. 
  \end{proof}

  \subsection{Proof of Theorem \ref{thm:GESn6} }
  With the results introduced in last section we are able to prove Theorem \ref{thm:GESn6}.
  We start with a consequence of the results presented in the previous subsection.

 \begin{teore}\label{teocomp} Suppose that $(\ub_{m})$ is any minimizing sequence for \eqref{S} with $\ub_m\geq0$. Then, up to translation and dilation, it is relatively compact in $\D$; i.e., there is a subsequence $(\ub_{m_{j}})$ and sequences $(R_{j})\subset \R$ and $(y_{j})\subset \R^{6}$ such that
 \begin{equation*}
 \vb_{j}(x):=
 R_{j}^{-2}\ub_{m_{j}}\left(R_{j}^{-1}(x-y_{j})\right), 
 \end{equation*}
 converges strongly in $\D$ to some $\vb$, which is a minimizer for \eqref{S}. 
 \end{teore}

  \begin{proof}
  Let $(\ub_{m})\subset \D$ be any minimizing sequence of \eqref{S}  with $\ub_m\geq0$, that is,
  \begin{equation}\label{uminseqS}
      \lim_{m\to \infty}K(\ub_{m})=S \qquad\mathrm{and}\qquad P(\ub_{m})=\int F(\ub_{m})\;dx=1.\\
  \end{equation}
  
   \noindent {\bf Claim 1.} 
There are sequences $(R_{m})\subset \R$ and $(y_{m})\subset \R^{6}$  such that  \begin{equation}\label{vmscalng6}
 \vb_{m}(x):=
 R_{m}^{-2}\ub_{m}\left(R_{m}^{-1}(x-y_{m})\right), 
 \end{equation}
 satisfies
 \begin{equation}\label{normliz}
 \sup_{y\in \R^{6}}\int_{B(y,1)}F(\vb_m(x))\; dx=   \int_{B(0,1)}F(\vb_{m})\; dx=\frac{1}{2}. 
 \end{equation}
 
 To prove this let us consider the following scaling
 \begin{equation*}
 \vb_{m}^{R,w}(x):=
 R^{-2}\ub_{m}\left(R^{-1}(x-w)\right), \qquad R>0, \quad w\in \R^{6}.
 \end{equation*}
 From Remark \ref{KPinvarice6} we have 
 $K(\vb_{m}^{R,w})=K(\ub_{m})$ and $P(\vb_{m}^{R,w})=P(\ub_{m})=1$. Denote by $Q_{m}^{R,w}(t)$ the concentration function corresponding to $F(\vb_{m})$, i.e., 
 \begin{equation*}
    Q_{m}^{R,w}(t):= \sup_{y\in \R^{6}}\int_{B(y,t)}F(\vb_{m}^{R,w}(x))\; dx. 
 \end{equation*}
 A change of variable gives that $Q_{m}(t/R)=Q_{m}^{R,w}(t)$, for all $t\geq 0$ and $w\in \R^{6}$, where, as in Lemma \ref{QmRattain},
 $$
 Q_m(t)=\sup_{y\in\R^{6}}\int _{B(y,t)} F(\ub_{m})\;dx.
 $$
 In particular, $Q_{m}(1/R)=Q_{m}^{R,w}(1)$ for all $m$.  Since for each $m$, $Q_{m}$ is a non-decreasing function with $Q_{m}(0)=0$ and $\lim_{t\to \infty}Q_{m}(t)=1$ we have
 \begin{equation*}
     \lim_{R\to 0^{+}}Q_{m}^{R,w}(1)=\lim_{R\to 0^{+}}Q_{m}\left(1/R\right)=1. 
 \end{equation*}
 Hence, for each $m$  we may  choose a number $R_{m}>0$ such that
 \begin{equation}\label{Qmt}
  Q_{m}^{R_{m},w}(1)=Q_{m}\left(1/R_{m}\right)=\frac{1}{2}, \qquad \mbox{for any}\; w\in \R^{6},   
 \end{equation}
that is,
 \begin{equation}\label{supFv}
      \sup_{y\in \R^{6}}\int_{B(y,1)}F(\vb_{m}^{R_m,w}(x))\; dx=Q_{m}^{R_{m},w}(1)=\frac{1}{2}, \qquad\mbox{for any}\; w\in \R^{6}.
 \end{equation}
 
 On the other hand, since $\int F(\vb_m^{R_m,w})=1$ and $\vb_m^{R_m,w}\geq0$,
  Lemma \ref{QmRattain} implies that there is $y_{m}\in \R^{6}$ such that
 \begin{equation*}
 \begin{split}
   \sup_{y\in \R^{6}}\int_{B(y,1)}F\left(R_{m}^{-2}\ub_{m}\left(R_{m}^{-1}(x-w)\right)\right)\;dx  &= \sup_{y\in \R^{6}}\int_{B(y,1)}F(\vb_m^{R_m,w}(x))\; dx\\
     &= \int_{B(y_{m},1)}F(\vb_m^{R_m,w}(x))\; dx\\
      &=\int_{B(0,1)}F\left(R_{m}^{-2}\ub_{m}\left(R_{m}^{-1}(z+y_{m}-w)\right)\right)\;dz,
 \end{split}
 \end{equation*}
where we have used the change of variables $x=z+y_{m}$.  Thus, taking $w=2y_{m}$ in the above equality and using \eqref{supFv} we obtain
\begin{equation*}
    \begin{split}
      \int_{B(0,1)}F\left(R_{m}^{-2}\ub_{m}\left(R_{m}^{-1}(z-y_{m})\right)\right)\;dz  &=\sup_{y\in \R^{6}}\int_{B(y,1)}F\left(R_{m}^{2}\ub_{m}\left(R_{m}^{-1}(x-2y_{m})\right)\right)\;dx\\
      &=Q_{m}^{R_{m},2y_{m}}(1)\\
      &=\frac{1}{2},
    \end{split}
\end{equation*}
which is the second equality in \eqref{normliz}. The first one also follows in view of \eqref{supFv}.\\
 
Next, from Remark \ref{KPinvarice6} and Claim 1 we have that  $(\vb_{m})$ is a minimizing sequence for \eqref{S} with $\vb_m\geq0$, which means that
 \begin{equation}\label{vminseqS}
      \lim_{m\to \infty}K(\vb_{m})=S \qquad\mathrm{and}\qquad P(\vb_{m})=\int F(\vb_{m})\;dx=1, \quad \mbox{for all} \;m\in\N. 
  \end{equation}
 In particular, $(\vb_{m})$ is uniformly bounded in $\D$. Then, there exist a subsequence, still denoted by $(\vb_{m})$, and $\vb\in \Dz $ such that
 \begin{equation}\label{weakconvm}
  \vb_{m} \rightharpoonup \vb, \quad \mathrm{ in }\quad \Dz.   
 \end{equation}
 It follows from Remark \ref{umnoneg} that $\vb\geq \mathbf{0}$.  
 
 Define the sequences of measures $(\mu_{m})$ and $(\nu_{m})$ by
 \begin{equation}\label{defmunum}
   \mu_{m}=\sum_{k=1}^{l}\gamma_{k}|\nabla v_{km}|^{2}\;dx,\qquad\mathrm{and}\qquad\nu_{m}=F(\vb_{m})\;dx.   
 \end{equation}
 
 From \eqref{vminseqS} we have that $(\nu_{m})$ is a probability sequence of measures. Then by Lemma \ref{CCLI} we know that, up to a subsequence,  one of the three cases occur:  \textit{vanishing}, \textit{dichotomy}, or \textit{compactness}. We will show that neither vanishing nor dichotomy occur. \\

   \noindent {\bf Claim 2.} Vanishing does not occur.
 
This follows immediately from \eqref{normliz} because
 \begin{equation*}
     \lim_{m\to\infty}\sup_{y\in \R^{6}}\nu_{m}(B(y,1))\geq\frac{1}{2}.\\
 \end{equation*}

 \noindent {\bf Claim 3.} Dichotomy does not occur.
 
In fact, suppose by contradiction that  dichotomy occurs. Then,  there is $\lambda\in (0,1)$ such that for any $\epsilon>0$ there exist a number $R>0$ and a sequence $(x_{m})$ with the property: given $R'>R$ and $m$ sufficiently large,
      \begin{equation}\label{numdicho}
          \nu_{m}(B(x_{m},R))\geq \lambda-\epsilon, \qquad \nu_{m}(\R^{6}\setminus B(x_{m},R'))\geq 1-\lambda-\epsilon.
      \end{equation}
For  $m$ (large) fixed and a given $\delta>0$, Corollary \ref{locGSI} implies that choosing $\rho$ such that $R<\rho<R'$ with $\frac{\rho}{R'}\leq C(\delta)$ and $\frac{R}{\rho}\leq C(\delta)$ we obtain 
 \begin{equation*}
     \int_{B(x_{m},R)}F(\vb_{m})\;dx\leq S^{-\frac{3}{2}}\left[\sum_{k=1}^{l}\int_{B(x_{m},\rho)}\gamma_{k}|\nabla v_{km}|^{2}\;dx+\delta K(\vb_{m}) \right]^{\frac{3}{2}},
 \end{equation*}
 \begin{equation*}
     \int_{\R^{6}\setminus B(x_{m},R')}F(\vb_{m})\;dx\leq S^{-\frac{3}{2}}\left[\sum_{k=1}^{l}\int_{\R^{6}\setminus B(x_{m},\rho)}\gamma_{k}|\nabla v_{km}|^{2}\;dx+(2\delta+\delta^{2}) K(\vb_{m}) \right]^{\frac{3}{2}},
 \end{equation*}
 These inequalities combined with \eqref{numdicho} lead to
 \begin{equation}\label{ineqSKdelt}
     S\left[(\lambda-\epsilon)^{\frac{2}{3}}+(1-\lambda-\epsilon)^{\frac{2}{3}}\right]\leq K(\vb_{m})+(3\delta+\delta^{2})K(\vb_{m}). 
 \end{equation}
 According to \eqref{vminseqS}, the right-hand side of \eqref{ineqSKdelt} is bounded by $K(\vb_{m})+(3\delta+\delta^{2})M$, for some positive constant $M$ independent of $m$. Thus, as $\epsilon,\delta \to 0$  and $m\to \infty$ we obtain
 \begin{equation}\label{ineqS}
    S\left[\lambda^{\frac{2}{3}}+(1-\lambda)^{\frac{2}{3}}\right]\leq S,
 \end{equation}
 that is, $\lambda^{\frac{2}{3}}+(1-\lambda)^{\frac{2}{3}}\leq1$. This is a contradiction with the fact that   $\lambda^{\frac{2}{3}}+(1-\lambda)^{\frac{2}{3}}>1$ for $\lambda\in (0,1)$. Hence, dichotomy does not occur.\\

As a consequence of Lemma \ref{CCLI} there is a sequence $(x_{m})\subset \R^{6}$, such that for any $\epsilon>0$ there exists a  positive number $R$  with
 \begin{equation}\label{compsnum}
     \nu_{m}(B(x_{m},R))\geq 1-\epsilon,\qquad \mbox{for all}\; m. 
 \end{equation}
 
  \noindent {\bf Claim 4.}  
The sequence $(\nu_{m})$ is  uniformly tight. 

In fact, we start claiming that $B(x_{m},R)\cap B(0,1)\neq \emptyset$, for all $m$.  Suppose the contrary, that is, there exists $m_{0}$ such that $B(x_{m_{0}},R)\cap B(0,1)= \emptyset$.   Taking $0<\epsilon<\frac{1}{2}$ in \eqref{compsnum} we have
 \begin{equation*}
     \int_{B(x_{m_{0}},R)}F(\vb_{m_0})\;dx> \frac{1}{2}. 
 \end{equation*}
 This combined with the normalization condition  \eqref{normliz} lead to
 \begin{equation*}
     \int F(\vb_{m_{0}})\;dx\geq \int_{B(x_{m_{0}},R)}F(\vb_{m_0})\;dx+\int_{B(0,1)} F(\vb_{m_{0}})\;dx>\frac{1}{2}+\frac{1}{2}=1,
 \end{equation*}
 which contradicts  \eqref{vminseqS}. Hence, the claim follows. 
 
 Next, because  $B(x_{m},R)\subset B(0,2R+1)$, for all $m$, \eqref{compsnum} yields
 \begin{equation*}\label{Ball2R}
     \nu_{m}(B(0,2R+1))\geq 1-\epsilon,\qquad \forall m. 
 \end{equation*}
Consequently, since $(\nu_{m})$ is a sequence of probability measures,
 \begin{equation*}
     \nu_{m}\left(\R^{6}\setminus \overline{B(0,2R+1)}\right)=1-\nu_{m}({B(0,2R+1)})\leq \epsilon, \quad \mbox{for all}\; m,  
 \end{equation*}
 that is, $(\nu_{m})$ is a uniformly tight sequence, as claimed.\\

  \noindent {\bf Claim 5.}  Up to a subsequence, $(\nu_{m})$ converges weakly to some $\nu\in\mathcal{M}_{+}^{1}(\R^{6})$.
 
Indeed, first note that for any $f\in \mathcal{C}_{c}(\R^{6})$, 
 \begin{equation*}
     \left|\int f\; d\nu_{m}\right|\leq \|f\|_{L^{\infty}}\nu_{m}\left(\R^{6}\right)=\|f\|_{L^{\infty}}<\infty. 
 \end{equation*}
 Thus, from Lemma \ref{vagbndconv},  there is $\nu \in \mathcal{M}_{+}^b(\R^{6}) $ such that, up to a subsequence, $\nu_{m}\overset{\ast}{\rightharpoonup} \nu$ in  $\mathcal{M}_{+}^b(\R^{6})$.  The uniform tightness of $(\nu_m)$ then implies that $\nu_m\rightharpoonup\nu$ weakly in $\mathcal{M}_{+}^b(\R^{6})$ (see for instance \cite[Theorem 30.8]{bauer2011measure}), i.e.,
 \begin{equation}\label{weakconvnum}
    \int f\;d\nu_{m} \to \int f\;d\nu ,\quad \mbox{for any}\quad f\in\mathcal{C}_b(\R^6).
 \end{equation}
In particular,   by taking $f\equiv1$, we obtain
\begin{equation}\label{nu=1}
    \nu\left(\R^{6}\right)=\lim_{m\to \infty}\nu_m(\R^6)=1,
\end{equation}
from which we deduce that $\nu\in\mathcal{M}_{+}^1(\R^{6})$.\\

Next, because $(K(\vb_{m}))$ is uniformly bounded it follows that $(\mu_{m})$ is also  vaguely bounded. Then, up to a subsequence, there exists $\mu\in \mathcal{M}_{+}^{b}(\R^{6})$ such that
 \begin{equation}\label{vagueconvmum}
  \mu_{m}\overset{\ast}{\rightharpoonup} \mu \quad \mathrm{in}\quad  \mathcal{M}_{+}^{b}(\R^{6}). 
 \end{equation}
 In particular we have $\mu(\R^6)\leq \liminf_{m\to \infty}\mu_m(\R^6)$.
 
 Now, \eqref{weakconvm}, \eqref{weakconvnum} and \eqref{vagueconvmum} allow us to invoke Lemma \ref{thm:CCLII} to obtain
 \begin{equation}\label{descomunu}
          \mu\geq \sum_{k=1}^{l}\gamma_{k}|\nabla v_{k}|^{2}\;dx+\sum_{j\in J}\mu_{j}\delta_{x_{j}}\quad\mathrm{and} \quad \nu=F(\vb)\; dx+\sum_{j\in J}\nu_{j}\delta_{x_{j}},
      \end{equation}
      for some family  $\{x_{j}\in \R^{6}:j\in J\}$, $J$ countable, and  $\mu_{j},\nu_{j}$ non-negative numbers satisfying
      \begin{equation}\label{Sobvjnuj}
          \nu_{j}\leq S^{-\frac{3}{2}}\mu_{j}^{\frac{3}{2}}, \qquad \mbox{for any}\; j \in J.
      \end{equation}
     with $\sum_{j\in J}\nu_j^{\frac{2}{3}}$ convergent.
Consequently,  \eqref{GCSIS}, \eqref{nu=1} and \eqref{Sobvjnuj} give
\begin{equation}\label{Sineq}
\begin{split}
        S=\liminf_{m\to \infty}\mu_{m}(\R^{6})&\geq \mu\left(\R^{6}\right)\\
        &\geq K(\vb)+\sum_{j\in J}\mu_{j}\\
        &\geq S\left[P(\vb)^{\frac{2}{3}}+\sum_{j\in J}\nu_{j}^{\frac{2}{3}}\right]\\
        &>S\left[P(\vb)+\sum_{j\in J}\nu_{j}\right]^{\frac{2}{3}}\\
        &=S\left[\nu\left(\R^{6}\right)\right]^{\frac{2}{3}}\\
        &=S,
\end{split}
\end{equation}
where we also have used that $\lambda \mapsto \lambda^{2/3}$ is a strictly concave function. Thus, all inequalities in \eqref{Sineq} are indeed equalities. But by the strictly concavity of the function $\lambda \to \lambda^{2/3}$, for \eqref{Sineq} to be an equality at most one of the terms $P(\vb)$ or $\nu_{j}$, $j\in J$, must be different from zero.\\

  \noindent {\bf Claim 6.} We claim  that $\nu_{j}=0$ for all $j\in J$. 
  
Otherwise, assume   $\nu_{j_{0}}\neq 0$ for some $j_{0}\in J$. Then from the above discussion, \eqref{nu=1} and the decomposition \eqref{descomunu} we obtain  $\nu=\nu_{j_0}\delta_{x_{j_{0}}}$, and then
\begin{equation}\label{nuj0}
    1=\nu(\R^6)=\nu_{j_{0}}. 
\end{equation}
The normalization condition \eqref{normliz} gives
\begin{equation*}
    \frac{1}{2}\geq \int_{B(x_{j_{0}},1)}F(\vb_{m})\;dx=\nu_{m}(B(x_{j_{0}},1)), \qquad \mbox{for all}\; m,
\end{equation*}
which leads to
\begin{equation*}
   \frac{1}{2}\geq\lim_{m\to \infty}\nu_{m}(B(x_{j_{0}},1))= \nu(B(x_{j_{0}},1))=\int_{B(x_{j_{0}},1)}d\nu=\nu_{j_{0}} ,
\end{equation*}
where the first equality is a consequence of   the weak convergence \eqref{weakconvnum}  (see, for instance, \cite[Theorem 30.12]{bauer2011measure}). But, this   contradicts  \eqref{nuj0}  and the claim is proved.\\

 Therefore, it must be the case that  $\nu=F(\vb)\;dx$ and from  \eqref{nu=1}
\begin{equation}\label{intFv=1}
   P(\vb)= \int F(\vb)\;dx=1,
\end{equation}
which means that $\vb\in \D$. 

It remains to prove that $K(\vb)=S$. From \eqref{intFv=1} and the definition of $S$ we know that $S\leq K(\vb)$. On the other hand, the lower semi-continuity of the weak convergence \eqref{weakconvm} gives $K(\vb)\leq \liminf_{m\to \infty}K(\vb_{m})=S$.  Hence, we conclude that $ K(\vb)=S=\lim_{m\to \infty}K(\vb_{m})$ and also that $\vb_{m}\to \vb$ strongly in $\D$. This finishes the proof. 
  \end{proof}
 
 Note that actually we  have proved the following:
 \begin{coro}\label{eleminiS}
 There exists $\vb\in \D $ satisfying $P(\vb)=1$ and $K(\vb)=C_{6}^{-\frac{2}{3}}$, where $C_{6}$ is the  best constant in the general critical Sobolev-type inequality \eqref{GCSI6}. 
 \end{coro} 
 
 Finally we are now in a position   to prove the existence of ground state solutions for  \eqref{syselip6}.

\begin{proof}[Proof of Theorem \ref{thm:GESn6}]
Let $\vb$ be the minimizer of \eqref{S} obtained in Theorem \ref{teocomp}.	From the Lagrange multiplier theorem,  there exists a constant $\Lambda$ such that
 \begin{equation}\label{Lagrange}
     2\gamma_{k} \int \nabla v_{k} \cdot \nabla g_{k}\;dx=\Lambda\int f_{k}(\vb)g_{k}\;dx,
 \end{equation}
 for any $\mathbf{g}\in \dot{\mathbf{H}}^{1}(\R^{6})$. 
By taking  $\mathbf{g}=\vb$ in \eqref{Lagrange} we promptly see that $\Lambda\neq0$. Now,
 define $\psib_{0}(x):=\frac{\Lambda}{2}\vb(x)$. By the above discussion $\psib_{0}$ is non-trivial. We will see that  $\psib_{0}$ is a ground state solution for \eqref{syselip6}. First of all, note    that $\boldsymbol{\psi}_{0}$ is a solution \eqref{syselip6}. Indeed, from \eqref{Lagrange} we have,  for any $\mathbf{g}\in \dot{\mathbf{H}}^{1}$,
 \begin{equation*}
    \begin{split}
        \gamma_{k} \int \nabla \psi_{0k} \cdot \nabla g_{k}\;dx&=\frac{\Lambda}{2} \gamma_{k} \int \nabla v_{k} \cdot \nabla g_{k}\;dx=\int \left(\frac{\Lambda}{2}\right)^{2}f_{k}(\vb)g_{k}\;dx=\int f_{k}(\psib_{0})g_{k}\;dx,
    \end{split} 
 \end{equation*}
 where we have used \eqref{fkhomog2} in the last equality.   
 
 Next, since  $\psib_{0}$ is a solution, from Remark  \ref{relatEIJ} it follows that $J(\psib_{0})=\frac{6^{\frac{3}{2}}}{2}I(\psib_{0})^{\frac{1}{2}}$.
 %\begin{equation*}
 %\label{ReltJI}
     %J(\psib_{0})=\frac{6^{\frac{3}{2}}}{2}I(\psib_{0})^{\frac{1}{2}}. 
% \end{equation*}
 On the other hand, according to Remark \ref{KPinvarice6} (i), $\vb$ is a minimizer of $J$ and since $J(\psib_{0})=J(\vb)$, so is  $\psib_{0}$. Consequently, $\psib_{0}$ is a ground state, as we required.
 %It follows from \eqref{ReltJI}  that $\psib_{0}$ minimize $I$, and therefore $\psib_{0}$ is a ground state, as we required.   
\end{proof}  

\begin{obs}
 We actually know the exact value of the Lagrange multiplier $\Lambda$. Indeed, since $\psib_{0}$ is a solution  of \eqref{syselip6}, from \eqref{KPrel} we have
 \begin{equation*}
     \left(\frac{\Lambda}{2}\right)^{2}K(\vb)=K(\psib_{0})=3P(\psib_{0})=3\left(\frac{\Lambda}{2}\right)^{3}P(\vb).
 \end{equation*}
 Hence, recalling that $K(\vb)=C_{6}^{-\frac{2}{3}}$ and $P(\vb)=1$ we deduce that 
  $\Lambda=\frac{2}{3}C_{6}^{-\frac{2}{3}}$. 
\end{obs}

\begin{coro}
  The inequality
  \begin{equation}\label{GNItype}
      P(\ub)\leq C_{6}^{opt}K(\ub)^{\frac{3}{2}},
  \end{equation}
  holds, for all $\ub \in 
 \D$, with the optimal constant $C_6^{opt}$ given by
  \begin{equation}\label{bestconstn6}
      C_{6}^{opt}=\frac{1}{3^{\frac{3}{2}}}\frac{1}{\E(\psib)^{\frac{1}{2}}},
  \end{equation}
  where $\psib$ is any ground state solution of \eqref{syselip6}. 
\end{coro}
\begin{proof}
In Remark \ref{KPinvarice6} we saw that \eqref{GNItype} holds with $C_6^{-1}=C_6^{opt}=\inf \{J(u); \,u\in \D\}$. Now if $\psib$ is any ground state of \eqref{syselip6}, Remark  \ref{relatEIJ} implies that
$$
C_6^{-1}=J(\psib)=\frac{6^{\frac{3}{2}}}{2}I(\psib)^{\frac{1}{2}}=3^{\frac{3}{2}}\E(\psib)^{\frac{1}{2}},
$$  
which is the desired.
\end{proof}

\begin{obs}
Note that all ground states of \eqref{syselip6} have the same energy. Therefore, the constant $C_6$ does not depend on the choice of the ground state.
\end{obs}

 \section{Blow-up results}\label{sec.blowupn56}
 This section aims to show the existence of blows-up solutions of  \eqref{system1}. To give the precise statement of our result we set
 \begin{equation*}
 \mathcal{G}:= 
 \begin{cases}
 \mathcal{G}_5(1,\boldsymbol{0}), \qquad if\quad n=5\\
 \mathcal{G}_6, \qquad if\quad n=6.
 \end{cases}
 \end{equation*}
 
 Thus the main result of this section is the following.

 \begin{teore}\label{thm:Blowupn6}
 	Assume that $\ub_{0}\in \mathbf{H}^1(\R^n)$ and let $\ub$ be the corresponding solution of \eqref{system1} defined in  the maximal time interval of existence, say  $I$. 
 	\begin{itemize}
 		\item[(i)] If $n=5 $ assume
 		\begin{equation}\label{desEgs1n=5}
 		Q(\mathbf{u}_{0})E(\mathbf{u}_{0})<Q(\boldsymbol{\psi})\E(\boldsymbol{\psi}),
 		\end{equation}
 		and 
 		\begin{equation}\label{desKgs1n=5}
 		Q(\mathbf{u}_{0})K(\mathbf{u}_{0})>Q(\boldsymbol{\psi})K(\boldsymbol{\psi}).
 		\end{equation}
 		\item[(ii)] If $n=6$ assume 
 		\begin{equation}\label{desEgs1n=6}
 		E(\mathbf{u}_{0})<\E(\boldsymbol{\psi})
 		\end{equation}
 		and
 		\begin{equation}\label{desKgs1n=6}
 		K(\mathbf{u}_{0})>K(\boldsymbol{\psi}).
 		\end{equation}
 	\end{itemize}
 	where $\mathcal{E}$ is the energy defined in \eqref{Efunc}
 	and  $\psib\in\mathcal{G} $. 
 	
 	Then, if $\ub_{0}$ is radially symmetric we have that $I$ is finite. 
 \end{teore}

 As we already said, to prove Theorem \ref{thm:Blowupn6} we follow closely the arguments in
   \cite{inui2018blow}.
 Let us start by introducing, for  $\varphi\in \mathcal{C}_{0}^{\infty}(\R^{n})$,
 \begin{equation*}
V(t)=\int \varphi(x)\left(\sum_{k=1}^{l}\frac{\alpha_{k}^{2}}{\gamma_{k}}|u_{k}|^{2}\right)\;dx.
\end{equation*}
Then, the solution $\ub$ of system \eqref{system1} satisfies
\begin{equation}\label{defR}
\begin{split}
    V'(t)&=2\sum_{k=1}^{l}\alpha_{k}\mathrm{Im}\int\nabla \varphi\cdot \nabla u_{k} \overline{u}_{k}\;dx-4\int\varphi(x)\mathrm{Im}\sum_{k=1}^{l}m_{k}f_{k}(\mathbf{u})\overline{u}_{k}\;dx\\
    &=:\mathcal{R}(t)-4\int\varphi(x)\mathrm{Im}\sum_{k=1}^{l}m_{k}f_{k}(\mathbf{u})\overline{u}_{k}\;dx.
\end{split}
\end{equation}

If $\ub_{0}$ is a radially symmetric function, so  is the corresponding  solution $\ub$. Hence, if in addition $\varphi$ is radially symmetric, by a direct calculation (see for instance \cite[Lemma 2.9]{Kavian} or \cite[Theorem 5.5]{NoPa2}) we can rewrite  $\mathcal{R}'$ as
 \begin{equation}\label{Rprim6}
\mathcal{R}'(t)=4\int \varphi''\left(\sum_{k=1}^{l}\gamma_{k}|\nabla u_{k}|^{2}\right)dx-\int\Delta^{2}\varphi\left(\sum_{k=1}^{l}\gamma_{k}|u_{k}|^{2}\right)dx-2\mathrm{Re}\int\Delta\varphi F\left(\mathbf{u}\right)\;dx.
\end{equation}

The approach used in \cite{inui2018blow} to prove the existence of blow-up solutions consists in getting a contradiction by working with $\mathcal{R}$ and $\mathcal{R}'$ instead of $V$ and $V''$.

 We start with two technical lemmas.
 
\begin{lem}\label{teclem}
Assume that $n\geq 1$. Let  $r=|x|$, $x\in \R^{n}$. Define, for a positive constant $c$,
\begin{equation}\label{definitionchi}
 \chi(r)=\left\{\begin{array}{cc}
r^{2},&0\leq r\leq 1,\\
c,& r\geq 3.
\end{array}\right.
\end{equation}
Assume also that $\chi''(r)\leq 2$ and $0\leq \chi'(r)\leq 2r$, for any $ r\geq 0$. Let $\chi_{R}(r)=R^{2}\chi\left(r/R\right)$. Then,
   If $r\leq R$,   
 \begin{equation}\label{laplachi1}
   \Delta\chi_{R}(r)=2n \qquad \mathrm{and}\qquad  \Delta^{2}\chi_{R}(r)=0.
 \end{equation}
 On the other hand, if $r\geq R$, then
 \begin{equation}\label{laplachi2}
 \Delta\chi_{R}(r)\leq C\qquad \mathrm{and}\qquad |\Delta^{2}\chi_{R}(r)|\leq \frac{C}{R^{2}},
 \end{equation}
 where $C$ is a constant independent of $R$.
\end{lem}
\begin{proof}
The proof is a straightforward computation.
\end{proof}

\begin{lem}\label{supercritcalcase}
Let $I$ be an open interval with $0\in I$. Let $a\in \R$, $b>0$ and $q>1$. Define $\gamma=(bq)^{-\frac{1}{q-1}}$ and $f(r)=a-r+br^{q}$, for $r\geq 0$. Let $G(t)$ a non-negative continuous  function such that $f\circ G\geq 0$ on $I$. Assume that $a<\left(1-\frac{1}{q}\right)\gamma$.
\begin{enumerate}
\item[(i)] If $G(0)<\gamma$, then $G(t)<\gamma$, $\forall t\in I$.
\item[(ii)] If $G(0)>\gamma$, then $G(t)>\gamma$, $\forall t\in I$.
\end{enumerate}
\end{lem}
\begin{proof}
See, for instance, \cite[Lemma 5.2]{beg}, \cite[Lemma 4.2]{Esfahani} or \cite[Lemma 3.1]{Pastor}.
\end{proof}

\subsection{Proof of Theorem \ref{thm:Blowupn6}}
In this section we will prove Theorem \ref{thm:Blowupn6}. Let us start by introducing the ``Pohozaev'' functional, 
 \begin{equation}\label{funcT}
   \mathcal{T}_{n}(\mathbf{u}(t))=K(\mathbf{u}(t))-\frac{n}{2} P(\mathbf{u}(t)), \qquad n=5,6.   
 \end{equation}
From the definition of the energy functional we may write
\begin{equation}\label{relTE}
   \mathcal{T}_{n}(\mathbf{u}(t))=\frac{n}{4}E(\mathbf{u}(t))-\left(\frac{n-4}{4}\right)K(\mathbf{u}(t))-\frac{n}{4}L(\mathbf{u}(t)).
\end{equation}

Our first result establishes that under the assumptions of Theorem \ref{thm:Blowupn6} the Pohozaev function is strictly negative.

\begin{lem}\label{lemTneg} Under the assumptions of Theorem \ref{thm:Blowupn6},
 there exists $\delta>0$ such that
\begin{equation*}
\mathcal{T}_{n}(\ub(t))\leq -\delta<0,\qquad t\in I.
\end{equation*}
\end{lem}
\begin{proof} 
 We follow the ideas presented in the proof of Theorem 1.3 in \cite{du2013blow}.  We will give the proof only in the cases $n=5$. The analysis for  $n=6$ follows exactly the same strategy using the results in Section \ref{sec.exigsn=6}.

We first note that by Lemma \ref{identitiesfunctionals} and the definition of the energy functional we obtain
 \begin{equation}\label{identKE}
     K(\boldsymbol{\psi})=5\E(\boldsymbol{\psi}).
 \end{equation}
Since $\psib\in\mathcal{G}_5(1,\boldsymbol{0})$ the functionals $\mathcal{Q}$ in \eqref{functionalQ} and $Q$ are the same. Therefore, from \ref{H6} and \eqref{GNI},
 \begin{equation}\label{estKxi0}
 \begin{split}
 K(\mathbf{u})&= E(\mathbf{u}_{0})-L(\mathbf{u})+2P(\mathbf{u})\leq E(\mathbf{u}_{0})+2\left|P(\mathbf{u})\right|\leq E(\mathbf{u}_{0})+2C_{5 }^{opt}Q(\ub_{0})^{\frac{1}{4}}K(\mathbf{u})^{\frac{5}{4}},
 \end{split}
 \end{equation}
 Now, in the notation of Lemma \ref{supercritcalcase}, if we take $G(t)=K(\mathbf{u}(t))$, $a=E(\mathbf{u}_{0})$, $b=2C_{5}^{opt}Q(\ub_{0})^{\frac{1}{4}}$  and $q=\frac{5}{4}$, then $\gamma=5\frac{Q(\boldsymbol{\psi})^{2}}{Q(\mathbf{u}_{0})}$ and from \eqref{estKxi0} 
 $f\circ G \geq 0$. Moreover, by using \eqref{identKE} a direct calculation gives
\begin{equation*}
    a<\left(1-\frac{1}{q}\right)\gamma \Longleftrightarrow Q(\mathbf{u}_{0})E(\mathbf{u}_{0})<Q(\boldsymbol{\psi})\E(\boldsymbol{\psi}),
\end{equation*}
\begin{equation*}
   G(0)>\gamma \Longleftrightarrow Q(\mathbf{u}_{0})K(\mathbf{u}_{0})>Q(\boldsymbol{\psi})K(\boldsymbol{\psi}).
\end{equation*}
 Hence, an application of Lemma \ref{supercritcalcase} yields
\begin{equation}\label{desK}
Q(\mathbf{u}_{0})K(\mathbf{u}(t))>Q(\boldsymbol{\psi})K(\boldsymbol{\psi}),\qquad t\in I.
\end{equation}
Thus, from \eqref{desEgs1n=5}, \eqref{identKE},  and \eqref{desK} we have
\begin{equation*}
    \begin{split}
      \frac{5}{4}E(\mathbf{u}(t))=\frac{5}{4}E(\mathbf{u}_{0})&<\frac{5}{4} \E (\boldsymbol{\psi})\frac{Q(\boldsymbol{\psi})}{Q(\mathbf{u}_{0})}=\frac{1}{4}K(\boldsymbol{\psi})\frac{Q(\boldsymbol{\psi})}{Q(\mathbf{u}_{0})}<\frac{1}{4}K(\mathbf{u}(t)).
    \end{split}
\end{equation*}
This combined with  \eqref{relTE} yields 
\begin{equation}\label{Tnegat}
    \mathcal{T}_{5}(\mathbf{u}(t))<0,\quad t\in I.  
\end{equation}

We claim that there exists $\sigma_{0}>0$ such that 
\begin{equation}\label{claim}
\mathcal{T}_{5}(\mathbf{u}(t))<-\sigma_{0}K(\mathbf{u}(t)), \qquad t\in I.
\end{equation}
Indeed, if $E(\ub_{0})\leq0$ from \eqref{relTE} we can promptly take $\sigma_{0}=\frac{1}{4}$. Now suppose $E(\ub_0)>0$ and assume by contradiction that \eqref{claim} does not hold. Then we can find sequences $(t_{m})\subset I$ and $(\sigma_{m})\subset\R_+$ with $\sigma_{m}\to 0$ such that
\begin{equation*}
-\sigma_{m}\frac{1}{4}K(\mathbf{u}(t_{m}))\leq \mathcal{T}_{5}(\mathbf{u}(t_{m}))<0.
\end{equation*}
 Thus, the last inequality and \eqref{relTE} gives
\begin{equation*}
\begin{split}
E(\mathbf{u}(t_{m}))&=\frac{4}{5}\mathcal{T}_{5}(\mathbf{u}(t_{m}))+\frac{1}{5}K(\mathbf{u}(t_{m}))+L(\mathbf{u}(t_{m}))\\
&\geq-\sigma_{m}\frac{1}{5}K(\mathbf{u}(t_{m}))+\frac{1}{5}K(\mathbf{u}(t_{m}))+L(\mathbf{u}(t_{m}))\\
&\geq(1-\sigma_{m})\frac{1}{5}K(\mathbf{u}(t_{m})).
\end{split}
\end{equation*}
From this, the conservation of the energy, \eqref{desK} and \eqref{identKE} we get
\begin{equation*}
\begin{split}
Q(\mathbf{u}_{0})E(\mathbf{u}_{0})&=Q(\mathbf{u}_{0})E(\mathbf{u}(t_{m}))\\
&\geq (1-\sigma_{m})\frac{1}{5}Q(\mathbf{u}_{0})
K(\mathbf{u}(t_{m}))\\
&>(1-\sigma_{m})\frac{1}{5}Q(\boldsymbol{\psi})K(\boldsymbol{\psi})\\
&=(1-\sigma_{m})Q(\boldsymbol{\psi})
\E(\boldsymbol{\psi}),
\end{split}
\end{equation*}
Taking $m\to \infty$ in the last inequality we obtain a contradiction with \eqref{desEgs1n=5}, so the claim is proved.

Finally note that \eqref{desK} gives $K(\ub(t))>K(\boldsymbol{\psi})\frac{Q(\boldsymbol{\psi})}{Q(\mathbf{u}_{0})}=:\epsilon_{0}$. Therefore the result follows immediately from \eqref{claim}.
\end{proof}
 
We are now in a position to prove Theorem \ref{thm:Blowupn6}.  
 \begin{proof}[Proof of Theorem \ref{thm:Blowupn6}]  Suppose that the maximal existence interval is  $I=(-T_{*},T^{*})$. We proceed by contradiction.  Without loss of generality assume that $T^{*}=+\infty$.  
 Using $\varphi(x)=\chi_{R}(|x|)$ with $\chi_R$ defined by \eqref{definitionchi}, from \eqref{defR} and \eqref{Rprim6} we can write
  $$
  \mathcal{R}(t)=2 \sum_{k=1}^{l}\alpha_{k}\int \nabla\chi_R \cdot\nabla u_{k}\overline{u}_{k}\;dx
  $$
  and
 \begin{equation*}
 \begin{split}
\mathcal{R}'(t)&=8\mathcal{T}_{n}(\mathbf{u})+4\int \left(\chi_{R}''-2\right)\left(\sum_{k=1}^{l}\gamma_{k}|\nabla u_{k}|^{2}\right)\;dx\\
&\quad-\int\Delta^{2}\chi_{R}\left(\sum_{k=1}^{l}\gamma_{k}|u_{k}|^{2}\right)\;dx-2\mathrm{Re}\int\left(\Delta\chi_{R}-2n\right) F\left(\mathbf{u}\right)\;dx\\
&=:8\mathcal{T}_{n}(\mathbf{u})+\mathcal{R}_{1}(t)+\mathcal{R}_{2}(t)+\mathcal{R}_{3}(t).
\end{split}
\end{equation*}
Here $R$ is seen as a parameter that will be chosen later.

Since from Lemma \ref{teclem} we have $\chi_{R}''(r)\leq 2$ for all $r\geq 0$, it follows that  $\mathcal{R}_{1}\leq 0$.
From \eqref{laplachi1} and the conservation of the charge, 
\begin{equation*}
    \mathcal{R}_{2}(t)\leq \int  |\Delta^{2} \chi_{R}|\left(\sum_{k=1}^{l}\gamma_{k}|u_{k}|^{2}\right)\;dx\leq C \int_{\{|x|\geq R\}}R^{-2}  \left(\sum_{k=1}^{l}\gamma_{k}|u_{k}|^{2}\right)\;dx\leq C R^{-2}Q(\mathbf{u}_{0}).
\end{equation*}
Also, \eqref{laplachi1}  and  Lemma \ref{estdifF} imply
\begin{equation*}
    \begin{split}
    \mathcal{R}_{3}&=-2\mathrm{Re}\int_{\{|x|\geq R\}}\left(\Delta\chi_{R}-2n\right) F\left(\mathbf{u}\right)\;dx \\
    &\leq C \int_{\{|x|\geq R\}}\left| \mathrm{Re}\,F\left(\mathbf{u}\right)\right|\;dx\\
    &\leq C \int_{\{|x|\geq R\}}\sum_{k=1}^{l}\left|u_{k} \right|^{3}\;dx\\
    &=C \sum_{k=1}^{l}\|u_{k}\|^{3}_{L^{3}(|x|\geq R)}.
    \end{split}
\end{equation*}
Recall that for any radial function $f\in H^1(\R^n)$ (see, for instance, \cite[equation (3.7)]{Ogawa})
$$
\|f\|^{3}_{L^{3}(|x|\geq R)}
\leq C  R^{-\frac{(n-1)}{2}}\|f\|^{\frac{5}{2}}_{L^{2}(|x|\geq R)}\|\nabla f\|^{\frac{1}{2}}_{L^{2}(|x|\geq R)}.
$$
 Hence, from Young's inequality  we can write, for any $\epsilon>0$, 
 \begin{equation*}
    \begin{split}
     \sum_{k=1}^{l}\|u_{k}\|^{3}_{L^{3}(|x|\geq R)}
    &\leq C  \sum_{k=1}^{l}R^{-\frac{(n-1)}{2}}\|u_{k}\|^{\frac{5}{2}}_{L^{2}(|x|\geq R)}\|\nabla u_{k}\|^{\frac{1}{2}}_{L^{2}(|x|\geq R)}\\
    &\leq C_{\epsilon}R^{\frac{-2(n-1)}{3}}Q(\mathbf{u}_{0})^{\frac{5}{3}}+2(n-4)\epsilon K(\mathbf{u}),
    \end{split}
\end{equation*}
where $C_\epsilon$ is a positive constant depending on $\epsilon$, $\alpha_{k}$, $\gamma_{k}$, and $\sigma_{k}$.

Gathering together the estimates for $\mathcal{R}_{1},\mathcal{R}_{2}$ and $\mathcal{R}_{3}$ we obtain
 \begin{equation}\label{Rprim1}
     \mathcal{R}'(t)\leq 8\mathcal{T}_{n}(\mathbf{u})+C R^{-2}Q(\mathbf{u}_{0})+2(n-4)\epsilon K(\mathbf{u})+C_{\epsilon}R^{\frac{-2(n-1)}{3}}Q(\mathbf{u}_{0})^{\frac{5}{3}},\qquad \epsilon>0.
 \end{equation}
 Assume  that $0<\epsilon<1$. Using \eqref{relTE} and Lemma \ref{lemTneg} we get 
\begin{equation*}
     \begin{split}
     \mathcal{R}'(t)&\leq 8(1-\epsilon)\mathcal{T}_{n}(\mathbf{u})+ 2n \epsilon |E(\mathbf{u}_{0})|+ C R^{-2}Q(\mathbf{u}_{0})+C_{\epsilon}R^{\frac{-2(n-1)}{3}}Q(\mathbf{u}_{0})^{\frac{5}{3}}\\
     &\leq -8(1-\epsilon)\delta+ 2n \epsilon |E(\mathbf{u}_{0})|+ C R^{-2}Q(\mathbf{u}_{0})+C_{\epsilon}R^{\frac{-2(n-1)}{3}}Q(\mathbf{u}_{0})^{\frac{5}{3}}.
     \end{split}
 \end{equation*}

 Now in the last inequality, we fix $R$ sufficiently large and choose $\epsilon$ sufficiently small such that $ \mathcal{R}'(t)\leq -2\delta$. 
Integrating this inequality on $[0,t)$ we obtain
\begin{equation}\label{Rt}
    \mathcal{R}(t)\leq -2\delta t +\mathcal{R}(0).  
 \end{equation}

On the other hand, from H\"{o}lder's inequality we deduce
\begin{equation}\label{modulR}
    \begin{split}
      |\mathcal{R}(t)|&\leq  2 \sum_{k=1}^{l}\alpha_{k}\int R|\chi'(|x|/R)|| \nabla u_{k}| |u_{k}|\;dx\\
        &\leq CR  \sum_{k=1}^{l}\alpha_{k}\|u_{k}\|_{L^{2}}\|\nabla u_{k}\|_{L^{2}}\\
        &\leq CRQ(\mathbf{u}_{0})^{\frac{1}{2}}K(\mathbf{u})^{\frac{1}{2}}.
    \end{split}
\end{equation}
We now may choose $T_{0}>0$ sufficiently large such that $\frac{\mathcal{R}(0)}{\delta}<T_{0}$. From this and \eqref{Rt},
\begin{equation}\label{Rnegti}
    \mathcal{R}(t)\leq -\delta t<0, \qquad t\geq T_{0}.
\end{equation}
Thus, \eqref{modulR} and \eqref{Rnegti}  lead to $\delta t\leq -\mathcal{R}(t)=|\mathcal{R}(t)|\leq CRQ(\mathbf{u}_{0})^{\frac{1}{2}}K(\mathbf{u})^{\frac{1}{2}}$, or equivalently,
\begin{equation}\label{Kquadrbon}
  K(\mathbf{u}(t))\geq C_0t^{2}, \qquad t\geq T_{0}
\end{equation}
for  some positive constant $C_0$.

Moreover, taking into account that $\epsilon$ is sufficiently small (less than $1/2$ is enough)  by \eqref{Rprim1} and \eqref{relTE} we get
\begin{equation}\label{Rprim3}
     \begin{split}
     \mathcal{R}'(t)&\leq 2nE(\mathbf{u})-2(n-4) K(\mathbf{u})+C R^{-2}Q(\mathbf{u}_{0})+(n-4) K(\mathbf{u})+CR^{\frac{-2(n-1)}{3}}Q(\mathbf{u}_{0})^{\frac{5}{3}}\\
 &\leq -(n-4) K(\mathbf{u})+ 2nE(\mathbf{u}_{0})+C R^{-2}Q(\mathbf{u}_{0})+CR^{\frac{-2(n-1)}{3}}Q(\mathbf{u}_{0})^{\frac{5}{3}},
     \end{split}
 \end{equation}
 where we have used  the conservation of the energy and the fact that $L(\ub)\geq 0$. We note that the last three terms in \eqref{Rprim3} do not depend on $t$. So, we may take $T_{1}>T_{0}$ such that
 \begin{equation*}
 \begin{split}
    C_0\frac{(n-4)}{2}T_{1}^{2} &\geq 2nE(\mathbf{u}_{0})+CR^{-2}Q(\mathbf{u}_{0})+CR^{\frac{-2(n-1)}{3}}Q(\mathbf{u}_{0})^{\frac{5}{3}},
 \end{split}
 \end{equation*}
 where $C_0$ is the constant appearing in  \eqref{Kquadrbon}. Thus,  \eqref{Kquadrbon}  and \eqref{Rprim3} give
\begin{equation*}
    \mathcal{R}'(t)\leq -\frac{(n-4)}{2}K(\mathbf{u}(t)), \quad t>T_{1}.
\end{equation*}
Now integrating the last inequality on $[T_{1}, t)$ gives
\begin{equation*}
\mathcal{R}(t)\leq -\frac{(n-4)}{2}\int_{T_{1}}^{t}K(\mathbf{u}(s))\; ds +\mathcal{R}(T_{1})\leq -\frac{(n-4)}{2}\int_{T_{1}}^{t}K(\mathbf{u}(s))\; ds.
\end{equation*}
Combining this  with \eqref{modulR} we get  
\begin{equation}\label{ineqK3}
\frac{(n-4)}{2}\int_{T_{1}}^{t}K(\mathbf{u}(s))\; ds\leq -\mathcal{R}(t)=|\mathcal{R}(t)|\leq CRQ(\mathbf{u}_{0})^{\frac{1}{2}}K(\mathbf{u})^{\frac{1}{2}}.
\end{equation}

Define $\displaystyle \eta(t):=\int_{T_{1}}^{t}K(\mathbf{u}(s))\;ds$ and $A:=\frac{(n-4)^{2}}{4C^{2}R^{2}Q(\mathbf{u}_{0})}$. From \eqref{Kquadrbon}, we have that $\eta(t)>0$ for $t>T_{1}$.  Thus \eqref{ineqK3} can be written as $A\leq\frac{\eta'(t)}{\eta^{2}(t)}$. Finally, taking $T_{1}<T'$ and integrating  on $[T',t)$ we obtain
\begin{equation*}
    A(t-T')\leq \int_{T'}^{t}\frac{\eta'(s)}{\eta^{2}(s)}\;ds=\frac{1}{\eta(T')}-\frac{1}{\eta(t)}\leq \frac{1}{\eta(T')},
\end{equation*}
or equivalently, 
\begin{equation*}
    0<\eta(T')\leq \frac{1}{ A(t-T')}.
\end{equation*}
Letting $t \to \infty$ we arrive to a contradiction. Hence the proof of Theorem \ref{thm:Blowupn6} is completed.
 \end{proof}

 \begin{obs}
 As we pointed out in Theorem \ref{thm:globalwellposH1}, if the inequality in \eqref{desKgs1n=5} is reversed then the solution exists globally in time. We believe if we reverse  inequality \eqref{desKgs1n=6} then the solution is also global. However, since this is the energy-critical case, much more efforts is needed. A possible technique to obtain the result is the one developed in \cite{killip}.
\end{obs}

\section*{Acknowledgement}

N.N. is partially supported by Universidad de Costa Rica, through the OAICE. 

A.P. is partially supported by CNPq/Brazil grants 402849/2016-7 and 303098/2016-3 and FAPESP/Brazil grant 2019/02512-5.


\begin{thebibliography}{20}
	

%\bibitem{Adams} R.A. Adams, J.J.F. Fournier, \emph{Sobolev Spaces}, Second Edition, Pure and Applied Mathematics, Elsevier/Academic Press, Amsterdam, 2003.
	

%\bibitem{Albert} J.P. Albert, \emph{Concentration compactness and stability of solitary-waves solutions to nonlocal equations}, Contemp. Math.  221 (1999), 1--30.
  
  
  \bibitem{bauer2011measure}  H. Bauer, \emph{Measure and integration theory: Translated from the German by Robert B. Burckel}, De Gruyter Studies in Mathematics, 26, Walter de Gruyter \& Co, 2001.
 
    
\bibitem{beg} P. B\'egout, \emph{Necessary conditions and sufficient conditions for global existence in the nonlinear Schr\"odinger equation}, Adv. Math. Sci. Appl. 12 (2002), 817--827.


\bibitem{ben1996extrema} A.K Ben-Naoum, C. Troestler and M. Willem, \emph{Extrema problems with critical Sobolev exponents on unbounded domains}, Nonlinear Anal. 26 (1996), no. 4, 823--833.


\bibitem{bogachev2007measure} V.I. Bogachev, \textit{Measure theory}, Springer-Verlag, 2007.



\bibitem{brezis1983relation} H. Brezis and  E. Lieb, \emph{A relation between pointwise convergence of function and convergence of functional}, Proc. Amer. Math. Soc.   88 (1983), no 3, 486-490.
	
    
%\bibitem{Brezis} H. Brezis, \textit{Functional Analysis, Sobolev Spaces and Partial Differential Equations}, Springer, New York, 2010.
    
    
	
%\bibitem{Buryak2} A.V. Buryak and Y.S.  Kivshar, \emph{Solitons due to second harmonic generation}, Phys. Lett. A.  197 (1995), no. 5/6, 407--412.


%\bibitem{Buryak} A.V. Buryak, P.  Di Tranpani, D.V.  Skryabin, and S. Trillo, \emph{Optical solitons due to quadratic nonlinearities: From basic physics to futuristic applications}, Phys. Rep.  370  (2002), no. 2, 63--235.


%\bibitem{burchard} A. Burchard and H. Hajaiej, \emph{Rearrangement inequalities for functionals with
%	monotone integrands}, J. Funct. Anal.  233 (2006), no. 2, 561--582.
    
\bibitem{Cazenave} T. Cazenave, \textit{Semilinear Schr\"odinger Equations}, Courant Lecture Notes in Mathematics, 10, American Mathematical Society,  Providence, RI, 2003. 
	
\bibitem{Colin} M. Colin, T.  Colin, and M.  Otha, \emph{Stability of solitary waves for a  system of  nonlinear Schr\"{o}dinger equations  with three waves interations}, Ann. Inst. H. Poincar\'e Anal. Non Lin\'eaire.  26 (2009), no. 6, 2211--2226.
	
    
\bibitem{Colin2} M. Colin, L. Di Menza, and J.C. Saut, \emph{Solitons in quadratic media},  Nonlinearity 29 (2016), no. 3, 1000--1035.
	
    
%\bibitem{Corcho} A. Corcho, S. Correia, F. Oliveira, and J.D. Silva, \emph{On a nonlinear Schr\"odinger system arising in quadratic media}, arXiv 1703.10509v2.


\bibitem{du2013blow} D. Du, Y. Wu, and K. Zhang, \emph{On Blow-up criterion for the Nonlinear Schr\"odinger Equation}, Discrete Contin. Dyn. Sist. 36 (2016), no. 7, 3639--3650.


				
\bibitem{Esfahani} A. Esfahani and A. Pastor, \emph{Sharp constant of an anisotropic Gagliardo-Nirenberg-type inequality and applications}, Bull. Braz. Math. Soc. 48 (2017), no. 1, 171--185.

    
\bibitem{EvansGas} L.C. Evans and R.F. Gariepy, \textit{Measure theory and fine properties of functions}, Studies in Advanced Mathematics,  CRC Press, Boca Raton, FL,  1992.



\bibitem{Evans1992measure} L.C. Evans, \textit{Weak convergence methods for nonlinear partial differential equations},  CBMS Regional Conference Series in Mathematics, 74, American Mathematical Society, Providence, RI, 1990.
    
    
\bibitem{Flucher}  M. Flucher and S. M\"uller, \emph{Concentration of low energy extremals},  Ann. Inst. H. Poincar\'e Anal. Non Lin\'eaire 16 (1999), no. 3, 269--298.
    



\bibitem{Folland} G.B. Folland, \textit{Real Analysis: Modern techniques and their applications}, Pure and Applied Mathematics, 19, 2nd ed, John Wiley \& Sons, Inc., New York , 1999.
  


%\bibitem{gil} D. Gilbarg and N. Trundiger, \textit{Nonlinear Partial Differential Equations of Second Order},  Classics in Mathematics, Springer, 1998.

%\bibitem{haj} H. Hajaiej, \emph{On the necessity of the assumptions used to prove Hardy-Littlewood and Riesz rearrangement inequalities}, Arch. Math. 96 (2011), no. 3, 273--280.
     
     
     
     
     
\bibitem{Hamano} M. Hamano, \emph{Global dynamics below the ground state for the quadratic Sch\"odinger system in 5D }, arXiv:1805.12245.
   
   
 \bibitem{hamano2019scattering} M. Hamano, \emph{Scattering for the quadratic nonlinear Schr\"{o}dinger system in  $\R^ 5$ without mass-resonance condition }, arXiv:1903.05880.
  
    
\bibitem{Hayashi3}  N. Hayashi, C. Li, and P.I.  Naumkin, \emph{On a system of  nonlinear Schr\"{o}dinger equations  in 2D}, Differential Integral Equations  24 (2011), no. 5/6, 417--434.
				
        
\bibitem{Hayashi}  N. Hayashi, T. Ozawa,  and K.  Tanaka,  \emph{On a system of nonlinear {S}chr\"odinger equations with quadratic interaction},  Ann. Inst. H. Poincar\'e Anal. Non Lin\'eaire. 30 (2013), no. 4, 661--690.
   
   
   
   
       
 \bibitem{hoshino2013analytic}  G. Hoshino and T. Ogawa, \emph{Analytic smoothing effect for a system of nonlinear Schr{\"o}dinger equations}, Differ. Equ. Appl. 5 (2013), no. 3, 395--408.  
 
 \bibitem{hoshino2015analytic}  G. Hoshino and T. Ogawa, \emph{Analytic smoothing effect for a system of Schr{\"o}dinger equations with two wave interaction}, Adv.  Differential Equations. 20 (2015), no. 7/8, 697--716.  
 
 \bibitem{hoshino2015analytic2}  G. Hoshino and T. Ozawa, \emph{Analytic smoothing effect for a system of Schr{\"o}dinger equations with three wave interaction}, J.  Math. Phys. 56 (2015), no. 9, 091513, 17 pp.  
 
 
 \bibitem{hoshino2017analytic}  G. Hoshino, \emph{Analytic smoothing effect for global solutions to a quadratic system of nonlinear Schr{\"o}dinger equations}, NoDEA Nonlinear Differential Equations  Appl. 24 (2017), no. 6, 62-1--62-17.  




\bibitem{inui2018blow} T. Inui, N. Kishimoto, and K. Nishimura, \emph{Blow-up of the radially symmetric solutions for the quadratic nonlinear {S}chr\"odinger system without mass-resonance}, arXiv preprint arXiv:1810.09153. 

\bibitem{kishimoto2018scattering} T. Inui,  N. Kishimoto, K. Nishimura, \emph{Scattering for a mass critical NLS system below the ground state with and without mass-resonance condition}, Discrete Contin. Dyn. Syst. 39 (2019), no 11, 6299--6353. 

\bibitem{iwabuchi2016ill}  T. Iwabuchi, T. Ogawa, and K. Uriya, \emph{Ill-posedness for a system of quadratic nonlinear Schr\"{o}dinger equations in two dimensions}, J.  Funct. Anal. 271 (2016), no. 1, 136--163.  



\bibitem{Kavian}  O. Kavian, \emph{A remark on the blowing-up solutions to the Cauchy problem for Nonlinear Schr\"{o}dinger Equations}, Trans. Am. Math. Soc. 299 (1987), no. 1, 193--203.  

\bibitem{killip} R. Killip and M. Visan,   \emph{The focusing energy-critical nonlinear Schr\"odinger equation
	in dimensions five and higher}, Amer. J. Math. 132 (2010), no. 2, 361--424.

\bibitem{Kiv} Y.S. Kivshar,  et al., \emph{Multi-component optical solitary waves}, Phys. A 288 (2000), no. 1/4, 152--173.

\bibitem{KoSa} K. Koynov and S. Saltiel, \emph{Nonlinear phase shift via multistep $\chi^2$ cascading}, Opt. Commun., 152 (1998), 96--100.   

  





% \bibitem{Leoni} G. Leoni, \emph{A first course in Sobolev spaces}, Graduate Studies in Mathematics, Vol. 105, American Mathematical Society, Providence, 2009. 
  
    
%\bibitem{LiHa} C. Li and N. Hayashi, N, \emph{Recent porgress on  nonlinear  Schr\"{o}dinger system with quadratic interactions}, The Scientific World Journal. 2014  (2014). 
 
 
% \bibitem{Li}  C. Li,  \emph{On a system of nonlinear  Schr\"{o}dinger equations and scale invariant spaces in 2D}, Differential  Integral Equations. 28 (2015), no. 3/4, 201--220.

   
    
%\bibitem{Linares}  F. Linares and G.   Ponce, \emph{Introduction to Nonlinear Dispersive Equations}, Universitex, Springer, New York, 2009.
    
\bibitem{Lieb}  E. Lieb and M.  Loss,  \emph{Analysis}, Graduate Studies in Mathematics, 14, 2nd ed, American Mathematical Society, Providence, RI, 2001.


\bibitem{lions1984concentration} P.L. Lions, \emph{The concentration compactness principle in the calculus of variations. The
locally compact case, I,}  Ann. Inst. H. Poincar\'e Anal. Non Lin\'eaire. 1  (1984), no 2, 109--145.

%\bibitem{Lions2} P.L. Lions, \emph{The concentration compactness principle in the calculus of variations. The
%locally compact case, II}, Ann. Inst. H. Poincar\'e Anal. Non Lin\'eaire. 1 (1984), no. 4, 223--283.

\bibitem{Lions3} P.L. Lions,  \emph{The concentration-compactness principle in the calculus of variations. The limit case, part 1.} Rev. Mat. Iberoamericana  1 (1985), no. 1, 145--201.
    

\bibitem{NoPa} N. Noguera and A. Pastor, \emph{On the dynamics of a quadratic Schr\"odinger system in dimension $n=5$}, Dyn. Partial. Differ. Equ. 17 (2020), no 1, 1--17.
 
\bibitem{NoPa2} N. Noguera and A. Pastor, \emph{On a system of Schr\"{o}dinger equations with general quadratic-type nonlinearities}, to appear in Commun. Contemp. Math. 2020.


	
	
\bibitem{Pastor}  A.  Pastor, \emph{Weak concentration and wave operator for a 3D coupled nonlinear Schr\"{o}dinger system}, J. Math. Phys.  56 (2015), no. 2, 021507, 18 pp.
	

\bibitem{Pastor2}  A. Pastor, \emph{On a three wave interaction  Schr\"{o}dinger systems with quadratic nonlinearities: Global well-posedness and standing waves},  Commun. Pure Appl. Anal. 18 (2019), no. 5, 2217--2242.


\bibitem{Ogawa} T. Ogawa and Y. Tsutsumi,   \emph{Blow-up of $H^{1}$ solutions for the Nonlinear Schr\"{o}dinger Equation }, J.  Differential Equations 92 (1991), no. 2,  317--330.

\bibitem{ogawa2015final} T. Ogawa and K. Uriya,   \emph{Final state problem for a quadratic nonlinear Schr{\"o}dinger system in two space dimensions with mass resonance }, J.  Differential Equations 258 (2015), no. 2,  483--503.

\bibitem{ozawa2013small} T. Ozawa and H. Sunagawa,   \emph{Small data blow-up for a system of nonlinear Schr{\"o}dinger equations }, J.  Math. Anal.  Appl. 399 (2013), no. 1,  147--155.


\bibitem{Struwe} M. Struwe,  \emph{Varational Methods: Application to Nonlinear Partial Differential Equations and Hamiltonian Systems}, Series of Modern Surveys in Mathematics, Springer-Verlag, 34,  2008.

\bibitem{talenti}  G. Talenti,   \emph{Best constant in Sobolev inequality},  Ann. Mat. Pura Appl. 110 (1976), no. 4, 353--372.


\bibitem{uriya2016final}  K. Uriya,   \emph{Final state problem for a system of nonlinear Schr{\"o}dinger equations with three wave interaction},  J. Evol. Equ. 16 (2016), no. 1, 173--191.




%\bibitem{Tao}  T. Tao,   \emph{Multilinear Weighted convolutions $L^{2}$  function, and applications to nonlinear dispersive equations}, Amer. J. Math. 123 (2001), no. 5, 839--908.


\bibitem{wang} B. Wang, Baoxiang, Z. Huo, C. Hao, Z. Guo, \textit{Harmonic analysis method for nonlinear evolution equations I}, World Scientific Publishing Co. Pte. Ltd., Hackensack, NJ, 2011.

\bibitem{wangyang}  H. Wang and Q. Yang  \emph{Scattering for the 5D quadratic NLS system
	without mass-resonance},  J. Math. Phys. 60, 121508 (2019).
  
  
%\bibitem{Yew2} A.C. Yew, A.R. Champneys, and P.J MacKenna, \emph{Multiple solitary waves due  to second-harmonic generation in quadratic media}, J. Nonlinear Sci. 9 (1999), no. 1, 33--52.        
 
 
% \bibitem{Yew} A.C. Yew, \emph{Multipulses of  nonlinearly coupled {S}chr\"odinger equation}, J. Differential Equations. 173 (2001), no. 1, 92--137.        

  
%\bibitem{Zhang} H. Zhang, \emph{Local well-posednees for a system of quadratic nonlinear {S}chr\"odinger equations in one or two dimensions}, Commun. Math. Meth.  App. Sci. 39 (2016), no. 14, 4257--4267.  
  
%\bibitem{Zhao} L. Zhao, F.  Zhao, and J. Shi, \emph{Higher dimensional solitary waves generated by secodn-harmonic generation in quadratic Media}, Calc. Var. 54 (2015), no. 3, 2657--2691.        
  		
\end{thebibliography}
\end{document}